\documentclass[final]{siamltex}

\newtheorem{remark}{Remark}[section]
\usepackage[cmex10]{amsmath}
\usepackage{graphics}           
\usepackage{graphicx}
\usepackage{subfigure}
\title{A Convergent Approximation of the Pareto Optimal Set for Finite Horizon Multiobjective Optimal Control  Problems (MOC)
Using Viability Theory \footnotemark[1]}

\author{A. Guigue \footnotemark[1]}

\begin{document}

\maketitle

\renewcommand{\thefootnote}{\fnsymbol{footnote}}
\footnotetext[1]{Department of Mathematics, The University of
British Columbia, Room 121, 1984 Mathematics Road, Vancouver, B.C.,
Canada, V6T 1Z2 (aguigue@math.ubc.ca).}

\begin{abstract}
The objective of this paper is to provide a convergent numerical
approximation of the Pareto optimal set for finite-horizon
multiobjective optimal control problems for which the objective space
is not necessarily convex. Our approach is based on
Viability Theory. We first introduce the set-valued return function
$V$ and show that the epigraph of $V$ is equal to the viability
kernel of a properly chosen closed set for a properly chosen
dynamics. We then introduce an approximate set-valued return
function with finite set-values as the solution of a multiobjective
dynamic programming equation. The epigraph of this approximate
set-valued return function is shown to be equal to the finite
discrete viability kernel resulting from the convergent numerical
approximation of the viability kernel proposed in
\cite{Cardaliaguet99_Via,Cardaliaguet00_Via}. As a result, the
epigraph of the approximate set-valued return function converges
towards the epigraph of $V$. The approximate set-valued return
function finally provides the proposed numerical approximation of
the Pareto optimal set for every initial time and state. Several
numerical examples are provided.
\end{abstract}



\pagestyle{myheadings} \thispagestyle{plain} \markboth{A.
GUIGUE}{CONVERGENT APPROXIMATIONS TO (MOC)}


\section{Introduction}

Many engineering applications, such as trajectory planning for
spacecraft~\cite{MOC_appl_6} and robotic
manipulators~\cite{Guigue10}, continuous casting of
steel~\cite{MOC_appl_5}, etc., can lead to an optimal control
formulation where $p$ objective functions ($p > 1$) need to be
optimized simultaneously. For a general optimization problem (GOP)
with a vector-valued objective function, the definition of an
optimal solution  requires the comparison between elements in the
objective space, which is the set of all possible values that can be
taken by the vector-valued objective function. This comparison is
generally provided by a binary relation, expressing the preferences
of the decision maker. In applications, it is common to consider the
binary relation defined in terms of a pointed convex cone $P \subset
\mathbf{R}^p$ containing the origin \cite{Yu74_moo}. However, in
this paper, for simplicity, we will only consider the case $P=
\mathbf{R}_+^p$, which yields the well-known Pareto optimality. The
resolution of (GOP) therefore consists of finding the set of Pareto
optimal elements in the objective space, or Pareto optimal set. In
general, this set cannot be obtained analytically and we have to
resort to numerical approximations. The main objective of this paper
is therefore to propose a convergent numerical approximation of the
Pareto optimal set for a general finite-horizon  multiobjective
optimal control problem~(MOC). By general, we mean that we do not
make any convexity assumption on the objective space $Y$ (or more
generally, on the set $Y + \mathbf{R}_+^p$). Indeed, in the case
where the objective space (or $Y + \mathbf{R}_+^p$) is convex,
simple methods such the weighting method can  be used to generate
the entire
Pareto optimal set (Theorem 3.4.4, \cite[p.\ 72]{NonlinearMOO1_book}).\\

When the objective space is not convex, very few approaches to find
the Pareto optimal set have been proposed. An important line of
research is to use evolutionary
algorithms~\cite{Book_EA,MOC_appl_7}, such as genetic algorithms.
Also, very recently, an approach~\cite{MOC_num_2} inspired  from the
$\epsilon$-constraint method  in nonlinear multiobjective
optimization \cite[pp.\ 85--95]{NonlinearMOO2_book} has been
developed for  multiobjective exit-time optimal control problems
where $P = \mathbf{R}_+^p$. In this approach, the $n$-dimensional
state is augmented by $p-1$ dimensions, which yields a new single
objective optimal control problem. The Pareto optimal set of the
original problem can be retrieved by inspecting the values of the
return function of this new problem. The return function of the new
problem is finally approximated by solving numerically the
corresponding $(n+p-1)$-dimensional Hamilton-Jacobi-Bellman equation
using a semi-Lagrangian ''marching'' method.\\


In this paper, instead of an exit-time optimal control problem, we
consider an optimal control problem over a finite horizon $[0,T]$.
The proposed approach fundamentally differs  from~\cite{MOC_num_2}.
Instead of augmenting the state space and solving the resulting
augmented Hamilton-Jacobi-Bellman equation, we define the set-valued
return function $V(\cdot,\cdot): [0,T] \times \mathbf{R}^n
\rightarrow 2^{\mathbf{R}^p}$ \cite{Guigue11,Guigue09} as the
set-valued map associating with each time $t \in [0,T]$ and state
$\mathbf{x} \in \mathbf{R}^n$ the set of  Pareto optimal elements in
the objective space $Y(t,\mathbf{x})$, where $Y(t,\mathbf{x})$ is
the set of all possible values that can be taken by the
vector-valued objective function for trajectories starting
at~$\mathbf{x}$ at time $t$. Hence, the Pareto optimal set for any
time $t$ and state $\mathbf{x}$ can be  obtained just by
evaluating~$V$ at $(t,\mathbf{x})$. We then derive a convergent
approximation of $V$ using Viability Theory
\cite{Book_Via,Cardaliaguet07_Via}. This approximation is a
set-valued map with finite set-values, called the approximate
set-valued return function. Hence, an approximation of the Pareto
optimal set $V(t,\mathbf{x})$ can be obtained  just by evaluating
the approximate set-valued return function at $(t,\mathbf{x})$. The
advantage of using Viability Theory is that it provides a framework
that allows to deal with problems with minimal regularity and
convexity assumptions. Hence, it is expected that the proposed
approach could be easily extended to more general classes of
problems than the one considered in this paper, e.g.,
problems with state constraints, etc..\\

More precisely, the first step in the proposed approach  is to show
that the epigraph of $V$, i.e., the graph of the set-valued map $V +
\mathbf{R}_+^p$, is equal to the viability kernel of a properly
chosen closed set for some properly chosen dynamics. The next step
is to introduce an approximate set-valued return function  as the
solution of a multiobjective dynamic programming equation. The
epigraph of this approximate set-valued return function is shown to
be equal to the finite discrete viability kernel resulting from the
convergent numerical approximation of the viability kernel proposed
in \cite{Cardaliaguet99_Via,Cardaliaguet00_Via}. From there, we
easily obtain that the epigraph of the approximate set-valued return
function converges in the sense of Painlev\'e-Kuratowski towards the
epigraph of $V$. The multiobjective dynamic programming equation
obtained is very similar to the one obtained in \cite{Guigue10_1},
where no proof of convergence was provided.\\

This paper is organized as follows. In \S \ref{s:P}, we detail the
class of multiobjective optimal control problems considered and
define  $V$.  In \S \ref{s:moogeneral}, we briefly discuss the
concept of optimality  in multiobjective optimization and present
several useful properties related to  Pareto optimal sets. In \S
\ref{s:charac}, we show that the epigraph of $V$ is equal to the
viability kernel of a properly chosen closed set for a properly
chosen dynamics. Following
\cite{Cardaliaguet99_Via,Cardaliaguet00_Via}, we then propose in \S
\ref{s:approx} a finite discrete approximation of this viability
kernel. In \S \ref{s:conv}, we show that the finite discrete
viability kernel resulting from this approximation is equal to the
epigraph of an approximate set-valued return function, defined as
the solution of a multiobjective dynamic programming equation. From
this multiobjective dynamic programming  equation, we derive in \S
\ref{s:numalgo} a numerical algorithm to compute the approximate
set-valued return function and therefore the approximate Pareto
optimal set at $(t,\mathbf{x})$. Some numerical examples
\cite{Guigue10_1}, for which the Pareto optimal set is analytically
known, are provided in \S \ref{s:ex} and some conclusions are
finally drawn in \S \ref{s:con}.

\section{A multiobjective finite-horizon optimal control problem (\cite{OptConTh8_book}}) \label{s:P}

In this paper, we will take for $\|\cdot\|$ in $\mathbf{R}^p$ and $\mathbf{R}^n$ the supremum norm. Let $\mathbf{B}$ be the closed unit ball.\\

Consider the evolution over a fixed finite time interval $I =
[0,T] \ ( 0 < T < \infty)$ of an autonomous dynamical system whose
$n$-dimensional state dynamics are given by a continuous function
$\mathbf{f}(\cdot,\cdot):  \mathbf{R}^n \times U
\rightarrow \mathbf{R}^{n}$, where the control space $U$ is a
nonempty compact subset of $\mathbf{R}^m$.
The function $\mathbf{f}(\cdot,\mathbf{u})$ is assumed to be
Lipschitz, i.e., some $K_{\mathbf{f}} > 0$ obeys
\begin{equation} \label{eq:lipf}
\forall \mathbf{u} \in U, \ \forall
\mathbf{x_1},\mathbf{x_2} \in \mathbf{R}^{n}, \ \|
\mathbf{f}(\mathbf{x_1},\mathbf{u}) -
\mathbf{f}(\mathbf{x_2},\mathbf{u}) \| \leq K_{\mathbf{f}} \|
\mathbf{x_1} - \mathbf{x_2} \|.
\end{equation}
We also assume that the function $\mathbf{f}$ is uniformly bounded, i.e., some
$M_\mathbf{f} > 0$ obeys
\begin{equation} \label{eq:fbounded}
\forall \mathbf{x} \in \mathbf{R}^n, \ \forall \mathbf{u} \in U, \
\| \mathbf{f}(\mathbf{x},\mathbf{u}) \| \leq M_\mathbf{f}.
\end{equation}
A control $\mathbf{u}(\cdot): I \rightarrow U$ is a bounded,
Lebesgue measurable function. The set of  controls
is denoted by $\mathcal{U}$. The continuity of $\mathbf{f}$ and the Lipschitz condition (\ref{eq:lipf})
guarantee that, given any $t \in I$, initial state $\mathbf{x} \in \mathbf{R}^n$, and
control $\mathbf{u}(\cdot) \in \mathcal{U}$,  the system
of differential equations governing the dynamical system,
\begin{equation} \label{eq:nec1}
\left \{ \begin{array}{lcl} \mathbf{\dot{x}}(s)& = & \mathbf{f}(\mathbf{x}(s),\mathbf{u}(s)), \ t \leq s  \leq T,\\
\mathbf{{x}}(t) & = & \mathbf{x},\end{array}\right.
\end{equation}
has a unique solution, called a trajectory and denoted $s \rightarrow \mathbf{x}(s;t,\mathbf{x},\mathbf{u}(\cdot))$.
Let $\mathrm{F}$ be the set-valued map defined from $\mathbf{R}^n$ to $\mathbf{R}^n$ by
$$\mathrm{F}(\mathbf{x}) = \bigcup_{\mathbf{u} \in U} \mathbf{f}(\mathbf{x},\mathbf{u}).$$

The cost of a trajectory over $[t,T], \ t \in I,$ is given by a
$p$-dimensional vector function $\mathbf{J}(\cdot,\cdot,\cdot): I
\times \mathbf{R}^n \times \mathcal{U} \rightarrow \mathbf{R}^p$,
\begin{equation} \label{eq:obj}
\mathbf{J}(t,\mathbf{x},\mathbf{u}(\cdot)) = \int_{t}^{T}
\mathbf{L}(\mathbf{x}(s;\mathbf{x},\mathbf{u}(\cdot)),\mathbf{u}(s))
\ \mathrm{d}s,
\end{equation}
where the $p$-dimensional vector function
$\mathbf{L}(\cdot,\cdot): \mathbf{R}^n \times U
\rightarrow \mathbf{R}^p$, called the running cost, is assumed to be continuous. For simplicity,  no terminal cost  is included in (\ref{eq:obj}). We assume that the function $\mathbf{L}$ is uniformly bounded, i.e., some
$M_\mathbf{L} \geq 0$ obeys
\begin{equation} \label{eq:Lbounded}
\forall \mathbf{x} \in \mathbf{R}^n, \ \forall \mathbf{u} \in U, \
\| \mathbf{L}(\mathbf{x},\mathbf{u}) \| \leq M_\mathbf{L},
\end{equation}
and that the function $\mathbf{L}(\cdot,\mathbf{u})$ satisfies a Lipschitz condition, i.e., some $K_\mathbf{L} \geq 0$ obeys
\begin{equation} \label{eq:LLipschitz}
\forall \mathbf{u} \in U, \ \forall
\mathbf{x_1},\mathbf{x_2} \in \mathbf{R}^n, \ \|
\mathbf{L}(\mathbf{x_1},\mathbf{u}) -
\mathbf{L}(\mathbf{x_2},\mathbf{u}) \| \leq K_\mathbf{L} \| \mathbf{x_1}
- \mathbf{x_2} \|.
\end{equation}
Let $\mathrm{L}$ be the set-valued map defined from $\mathbf{R}^n$ to $\mathbf{R}^p$ by
$$\mathrm{L}(\mathbf{x}) = \bigcup_{\mathbf{u} \in U} \mathbf{L}(\mathbf{x},\mathbf{u}).$$

The objective space $Y(t,\mathbf{x})$  for (MOC) is defined as the set of
all possible costs~(\ref{eq:obj}):
\begin{displaymath}
\ Y(t,\mathbf{x}) = \bigg\{
\mathbf{J}(t,\mathbf{x},\mathbf{u}(\cdot)), {\mathbf{u}(\cdot) \in
\mathcal{U}} \bigg\}.
\end{displaymath}
From  $(\ref{eq:Lbounded})$, it follows that the set $Y(t,\mathbf{x})$ is bounded
(by $M_\mathbf{L}T$), and also that $Y(t,\mathbf{x}) \subset \{-(T-t)M_\mathbf{L}\mathbf{1}\} + \mathbf{R}^p_+$.
However, the set $Y(t,\mathbf{x})$ is not necessarily closed.\\

The set-valued return function $V(\cdot,\cdot): I \times
\mathbf{R}^n  \rightarrow 2^{\mathbf{R}^p}$ for (MOC) is defined as
the set-valued map which associates with each time $t \in I$ and
initial state $\mathbf{x} \in \mathbf{R}^n$ the set of Pareto
optimal elements in the objective space $Y(t,\mathbf{x})$, where the
definition of a Pareto optimal element is postponed to \S
\ref{s:moogeneral}:
\begin{equation} \label{eq:setvalued}
V(t,\mathbf{x}) = \mathcal{E}(\mathrm{cl}(Y(t,\mathbf{x}))).
\end{equation}
The closure in (\ref{eq:setvalued}) is used to guarantee the existence of Pareto optimal elements (Proposition 3.5,
\cite{Guigue09}). Hence, $\forall t \in I, \ \forall \mathbf{x} \in \mathbf{R}^n, \ V(t,\mathbf{x}) \neq \emptyset$.\\
\begin{remark} \label{rem:valuefunction1dof} When $p=1$, \emph{(\ref{eq:setvalued})}
takes the form
$$V(t,\mathbf{x})  =  \bigg\{ \inf_{\mathbf{u}(\cdot) \in \mathcal{U}}  \int_{t}^{T}
L(\mathbf{x}(s;t,\mathbf{x},\mathbf{u}(\cdot)),\mathbf{u}(s)) \ \mathrm{d}s   \bigg\}.$$
Hence, $V(t,\mathbf{x}) = \{v(t,\mathbf{x})\},$ where $v(\cdot,\cdot)$ is the value function
for single objective optimal control problems \emph{\cite{OptConTh8_book,OptConTh5_book}}.
\end{remark}\\ \\
Finally, as $V(t,\mathbf{x}) \subset \mathrm{cl}(Y(t,\mathbf{x}))$,
we have \begin{equation} \label{eq:inclusion} V(t,\mathbf{x}) \subset \{-(T-t)M_\mathbf{L}\mathbf{1}\}  +
\mathbf{R}^p_+.\end{equation}

The  objective of this paper is to find a convergent approximation
to the Pareto optimal set $V(0,\mathbf{x_0})$ where $\mathbf{x_0}
\in \mathbf{R}^n$ is some given initial state.

\section{Multiobjective Optimization} \label{s:moogeneral}

For an optimization problem with a $p$-dimen-sional vector-valued objective function, the definition
of an optimal solution requires the comparison of any two elements $\mathbf{y_1},\mathbf{y_2}$ in the objective space, which is the set of all possible values that can
be taken by the vector-valued objective function. This
comparison is generally provided by a binary relation, expressing
the preferences of the decision maker. In applications, it is common
to consider the binary relation defined in terms of a pointed convex
cone $P \subset \mathbf{R}^p$ containing the origin \cite{Yu74_moo}.\\
\begin{definition} \label{def:binaryRelation}
Let $\mathbf{y_1},\mathbf{y_2} \in \mathbf{R}^p$. Then, $\mathbf{y_1} \preceq \mathbf{y_2}$  if and only if
$\mathbf{y_2} \in \mathbf{y_1} + P$.
\end{definition}\\

The binary relation in Definition \ref{def:binaryRelation}  yields  the definition
of generalized Pareto optimality.\\
\begin{definition} \label{def:minele}
Let $S$ be a nonempty subset of $\mathbf{R}^p$. An element
$\mathbf{y_1} \in S$ is said to be a generalized Pareto optimal
element of $S$ if and only if there is no $\mathbf{y_2} \in S \
(\mathbf{y_2} \neq \mathbf{y_1})$ such that $\mathbf{y_1} \in
\mathbf{y_2} + P,$ or equivalently, if and only if  there is no
$\mathbf{y_2}$ such that $\mathbf{y_1 }\in \mathbf{y_2} +
P\backslash\{\mathbf{0}\}$. The set of generalized Pareto optimal
elements of $S$ is called the generalized Pareto optimal set and is
denoted by $\mathcal{E}(S,P)$. When $P = \mathbf{R}^p_+$, the
generalized Pareto optimal elements are only referred to as Pareto
optimal elements, and the set $\mathcal{E}(S,P)$ is simply denoted
$\mathcal{E}(S).$
\end{definition}\\

An important role in this paper is played by  the external stability or domination property \cite[pp.~59-66]{Tanino88_moo}.\\

\begin{definition}[External stability]
A nonempty subset $S$ of $\mathbf{R}^p$ is said to be externally stable if and only
if $$ S \subset \mathcal{E}(S,P) + P.$$
\end{definition}
An immediate consequence of the external stability property is that $S + P  = \mathcal{E}(S,P) + P$. When $P$ is closed, a sufficient condition for a nonempty closed set $S$ to be externally stable is given in Proposition~\ref{prop:exter}. Note that this condition is also sufficient to guarantee the existence of
generalized Pareto optimal  elements.\\

\begin{proposition}[Theorem 3.2.10, \emph{\emph{\cite[p.\ 62]{Tanino88_moo}}}] \label{prop:exter}
Let $S$ be a nonempty closed subset of $\mathbf{R}^p$. If $P$ is closed and $S$ is $P$-bounded \emph{\cite[p.~52]{Tanino88_moo}}, i.e., $S^+ \cap - P = \{\mathbf{0}\},$ then $S$ is externally stable.\\
\end{proposition}

\begin{corollary} \label{cor:extstabcompactset}
Let $K$ be a nonempty compact subset of $\mathbf{R}^p$. If $P$ is closed, then $K$ is externally stable.
\end{corollary} \\

In this paper, for simplicity, we will only consider the case $P = \mathbf{R}^p_+.$

\section{Characterization of the set-valued return function} \label{s:charac}

In this section, we show that the epigraph of the set-valued return
function $V$, i.e., the graph of  the set-valued map $V +
\mathbf{R}^p_+$, is equal
to the viability kernel of a properly chosen closed set for some properly chosen dynamics.\\

Define the set-valued maps $\mathrm{FL^\sigma}$ from $\mathbf{R}^n$
to $\mathbf{R}^n \times \mathbf{R}^p$ by
$$\mathrm{FL^\sigma}(\mathbf{x}) = \overline{\mathrm{co}}\bigg(\bigcup_{\mathbf{u} \in U} \{(\mathbf{f}(\mathbf{x},\mathbf{u}),\sigma\mathbf{L}(\mathbf{x},\mathbf{u}))\}\bigg),$$ and
where $\overline{\mathrm{co}}(S)$ denotes the closure of the convex
hull of the set $S$. Observe that the set-valued map
$\mathrm{FL^\sigma}$ takes convex compact nonempty values. Moreover,
$\mathrm{FL^\sigma}$ is bounded by $M_{\mathrm{FL}} =
\max\{M_{\mathbf{f}},M_{\mathbf{L}}\}$ and Lipschitz with Lipschitz
constant $K_{\mathrm{FL}} =
 \max\{K_{\mathbf{f}},K_{\mathbf{L}}\}$. Denote $
\mathrm{FL^+}(\mathbf{x}) = \mathrm{FL^{+1}}(\mathbf{x})$ and $
\mathrm{FL^-}(\mathbf{x}) = \mathrm{FL^{-1}}(\mathbf{x})$. Define
finally the expanded set-valued map $\phi$ from $\mathbf{R} \times
\mathbf{R}^n \times \mathbf{R}^p$ to $\mathbf{R} \times \mathbf{R}^n
\times \mathbf{R}^p$ by

\begin{equation} \label{eq:expandedynamics}
\phi(t,\mathbf{x},\mathbf{z}) = \left \{ \begin{array}{lcl} \{1\} \times \mathrm{FL^-}(\mathbf{x})& \mathrm{if}  &  t < T,\\  {[}0,1{]} \times \overline{\mathrm{co}}(\mathrm{FL^-}(\mathbf{x}) \cup \{(\mathbf{0},\mathbf{0})\}) & \mathrm{if} & t \geq T.\end{array}\right.
\end{equation} \\
It is easy to see that $\phi$ is a Marchaud map (Definition 2.2, p.~184, \cite{Cardaliaguet99_Via}) bounded
by $\max \{1,M_{\mathrm{FL}}\}$.\\ \\

Consider now the differential inclusion
\begin{equation} \label{eq:diffinclusionexpanded}(\dot{t}(s),\mathbf{\dot{x}}(s),\mathbf{\dot{z}}(s)) \in \phi(t(s),\mathbf{x}(s),\mathbf{z}(s)) \ a.e.
 \ s \geq 0,\end{equation} and the closed set $\mathcal{H} = \{ (t,\mathbf{x},\mathbf{z}), \ t \in [0,T], \ \mathbf{x} \in \mathbf{R}^n, \mathbf{z} \in \{-(T-t)M_\mathbf{L} \mathbf{1}\}+ \mathbf{R}^p_+$.\\

\begin{proposition}  \label{prop:char}
The epigraph of the  set-valued map $V$ is equal to the viability
kernel of $\mathcal{H}$ for $\phi$, i.e.,
$$\mathrm{Epi}(V) = \mathrm{Viab}_\phi(\mathcal{H}),$$ where $\mathrm{Epi}(V)$ is defined as
$\mathrm{Epi}(V) = \mathrm{Graph}(V + \mathbf{R}^p_+).$
\end{proposition}
\begin{proof}
First, we prove the inclusion
$$ \mathrm{Viab}_\phi(\mathcal{H}) \subset \mathrm{Graph}(V + \mathbf{R}^p_+).$$
Take $(t,\mathbf{x},\mathbf{z}) \in \mathrm{Viab}_\phi(\mathcal{H})$. Hence, $(t,\mathbf{x},\mathbf{z}) \in \mathcal{H}$, or $t \in I$ and $\mathbf{z} \in  \{-(T-t)M_\mathbf{L} \mathbf{1}\}+ \mathbf{R}^p_+$.\\
\begin{itemize}
  \item Assume that $t = T$. Then, as  $V(T,\mathbf{x}) = \{\mathbf{0}\}$, it follows that $\mathbf{z} \in \mathbf{R}^p_+ = V(T,\mathbf{x}) + \mathbf{R}^p_+$.
  \item Assume that $t \in [0,T)$. Let $(t(\cdot),\mathbf{x}(\cdot),\mathbf{z}(\cdot))$ be a solution to (\ref{eq:diffinclusionexpanded}) with initial condition $(t,\mathbf{x},\mathbf{z})$ which remains in $\mathcal{H}$. By definition of $\phi$,
     $(\mathbf{x}(\cdot),\mathbf{z}(\cdot))$ is a solution
      to the differential inclusion
\begin{equation} \nonumber
 \left \{ \begin{array}{lcl} (\mathbf{\dot{x}}(s),\mathbf{\dot{z}}(s)) & \in & \mathrm{FL^-}(\mathbf{x}(s)) \ a.e. \ s \in [0, T-t], \\ \mathbf{x}(0) & = & \mathbf{x},\\
\mathbf{z}(0) & = & \mathbf{z},\\\end{array}\right.
\end{equation}
and $t(s) = s + t$. Let $s' = s + t$. For $s' \in [t,T]$, define $\mathbf{x'}(s') = \mathbf{x}(s'-t)$
and $\mathbf{z'}(s') = \mathbf{z}(s'-t).$ Then, $(\mathbf{x'}(\cdot),\mathbf{z'}(\cdot))$ is a solution
      to the differential inclusion
\begin{equation} \nonumber
 \left \{ \begin{array}{lcl} (\mathbf{\dot{x}'}(s),\mathbf{\dot{z}'}(s')) & \in & \mathrm{FL^-}(\mathbf{x'}(s')) \ a.e. \ s' \in [t, T], \\ \mathbf{x'}(t) & = & \mathbf{x},\\
\mathbf{z'}(t) & = & \mathbf{z},\\\end{array}\right.
\end{equation}
By the Relaxation Theorem (Theorem 2.7.2, \cite[p.~96]{OptConTh5_book}), for $\epsilon >0$, there exists $\mathbf{u}(\cdot) \in \mathcal{U}$ such that
$$\| \mathbf{x'}(\cdot) - \mathbf{x}(\cdot;t,\mathbf{x},\mathbf{u}(\cdot)) \| \leq \epsilon \ \mathrm{and} \
\|\mathbf{z'}(\cdot) - \mathbf{z}(\cdot;t,(\mathbf{x},\mathbf{z}),\mathbf{u}(\cdot))\|\leq \epsilon.$$
In particular, we get
$$ \mathbf{z'}(T) \leq \mathbf{z}(T;t,(\mathbf{x},\mathbf{z}),\mathbf{u}(\cdot)) + \epsilon \mathbf{1}.$$
Moreover, as $(s',\mathbf{x'}(s'),\mathbf{z'}(s')) \in \mathcal{H}$ for all $s' \in [t,T],$ we have
$$\mathbf{z'}(s') \in  \{-(T-s')M_\mathbf{L} \mathbf{1}\}+ \mathbf{R}^p_+,$$
or
$$\mathbf{z'}(T) \in  \mathbf{R}^p_+.$$
Hence,
\begin{equation} \nonumber
 \begin{array}{lcl}  \mathbf{z}(T;t,(\mathbf{x},\mathbf{z}),\mathbf{u}(\cdot)) & = & \mathbf{z} -  \displaystyle \int_t^{T} \mathbf{L}(\mathbf{x}(s';t,\mathbf{x},\mathbf{u}(\cdot)),\mathbf{u}(s')) \ \mathrm{d} s' \geq \mathbf{z'}(T) - \epsilon \mathbf{1} \geq -\epsilon \mathbf{1},\\  \end{array}
\end{equation}
or
$$ \mathbf{z} + \epsilon \mathbf{1} \geq \displaystyle \int_t^{T} \mathbf{L}(\mathbf{x}(s';t,\mathbf{x},\mathbf{u}(\cdot)),\mathbf{u}(s')) \ \mathrm{d} s',$$
which implies that $\mathbf{z} \in \mathrm{cl}(Y(t,\mathbf{x})) +
\mathbf{R}^p_+ = V(t,\mathbf{x}) + \mathbf{R}^p_+$ by external
stability.\newline
\end{itemize}

Second, we prove the inclusion
$$ \mathrm{Graph}(V +  \mathbf{R}^p_+) \subset \mathrm{Viab}_\phi(\mathcal{H}).$$
Take $(t,\mathbf{x},\mathbf{z}) \in \mathrm{Graph}(V +  \mathbf{R}^p_+)$. Hence, $t \in [0,T]$, from  (\ref{eq:inclusion}), $\mathbf{z} \in \{-(T-t)M_\mathbf{L}\mathbf{1}\} +   \mathbf{R}^p_+$.\\
\begin{itemize}
  \item Assume that $t = T$. Then, $\mathbf{z} \in \mathbf{R}^p_+$. Therefore, $(t(\cdot),\mathbf{x}(\cdot),\mathbf{z}(\cdot)) = (T,\mathbf{x},\mathbf{z})$ is a solution  to (\ref{eq:diffinclusionexpanded}) viable in $\mathcal{H}$.
  \item Assume that $t \in [0,T)$. We have $\mathbf{z} = \mathbf{z'} + \mathbf{d}$, where $\mathbf{z'} \in V(t,\mathbf{x})$ and $\mathbf{d} \in \mathbf{R}^p_+$. By definition of  $V(t,\mathbf{x})$, there exists a sequence $\mathbf{u}_n(\cdot) \in \mathcal{U}$ such that
\begin{equation} \label{eq:conv} \lim_{n \rightarrow +\infty} \displaystyle \int_t^T \mathbf{L}(\mathbf{x}_n(s;t,\mathbf{x},\mathbf{u}_n(\cdot)),\mathbf{u}_n(s)) \ \mathrm{d}s = \mathbf{z'}.\end{equation} Using the Compactness of Trajectories theorem (Theorem 2.5.3, \cite[p.~89]{OptConTh5_book}) and by passing to a subsequence if necessary, we can therefore assume that
      there exists $(\mathbf{x}(\cdot),\mathbf{z}(\cdot))$ solution on $[t,T]$ to the differential inclusion
$$(\mathbf{\dot{x}}(s'),\mathbf{\dot{z}}(s')) \in \mathrm{FL^+}(\mathbf{x}(s')) \ a.e. \ s' \in [t,T]$$
with initial condition $(\mathbf{x},\mathbf{0})$
such that $$\lim_{n \rightarrow +\infty} \|\mathbf{x}_n(\cdot;t,\mathbf{x},\mathbf{u}_n(\cdot)) - \mathbf{x}(\cdot)\| =  0,$$
      and \begin{equation} \label{eq:zlim} \lim_{n \rightarrow +\infty} \|\mathbf{z}_n(\cdot;t,(\mathbf{x},\mathbf{0}),\mathbf{u}_n(\cdot)) - \mathbf{z}(\cdot)\| =  0.\end{equation}
\end{itemize}
From (\ref{eq:conv}) and (\ref{eq:zlim}), we deduce that $ \mathbf{z}(T) = \mathbf{z'}.$\\

Define now $(\mathbf{x'}(\cdot),\mathbf{z'}(\cdot))$ as follows:
      \begin{equation} \nonumber
      (t(s),\mathbf{x'}(s),\mathbf{z'}(s)) = \left \{ \begin{array}{ll} (s+t,\mathbf{x}(s+t),\mathbf{z}- \mathbf{z}(s+t)) & s \in [0,T-t],\\
      (T,\mathbf{x}(T),\mathbf{d})& s > T-t. \end{array}\right.
      \end{equation}
As $\mathbf{z'}(T-t) = \mathbf{z}- \mathbf{z}(T) = \mathbf{z}- \mathbf{z'} = \mathbf{d},$ it follows that $(t(\cdot),\mathbf{x'}(\cdot),\mathbf{z'}(\cdot))$ is a solution of (\ref{eq:diffinclusionexpanded}). It
remains to check that $(t(s),\mathbf{x'}(s),\mathbf{z'}(s))\in \mathcal{H}$ for all $s \geq 0$. For $s > T-t$,
as $\mathbf{d} \in \mathbf{R}^p_+$, this is obvious. For $s \in [0,T-t]$, using the definition
of $\mathbf{z}$,  (\ref{eq:conv}), and (\ref{eq:zlim}), we have
$$
\mathbf{z'}(s) = \lim_{n \rightarrow +\infty} \displaystyle \int_t^T \mathbf{L}(\mathbf{x}_n(s;t,\mathbf{x},\mathbf{u}_n(\cdot)),\mathbf{u}_n(s)) \ \mathrm{d}s  + \mathbf{d}  -  \mathbf{z}_n(s+t;t,(\mathbf{x},\mathbf{0}),\mathbf{u}_n(\cdot)).
$$
As
$$\mathbf{z}_n(s+t;t,(\mathbf{x},\mathbf{0}),\mathbf{u}_n(\cdot)) =  \displaystyle \int_t^{s+t} \mathbf{L}(\mathbf{x}_n(s;t,\mathbf{x},\mathbf{u}_n(\cdot)),\mathbf{u}_n(s)) \ \mathrm{d}s,$$
we get
$$
\mathbf{z'}(s) =  \mathbf{d} + \lim_{n \rightarrow +\infty} \displaystyle \int_{s+t}^T \mathbf{L}(\mathbf{x}_n(s;t,\mathbf{x},\mathbf{u}_n(\cdot)),\mathbf{u}_n(s)) \ \mathrm{d}s.
$$
As
$$\int_{s+t}^T \mathbf{L}(\mathbf{x}_n(s;t,\mathbf{x},\mathbf{u}_n(\cdot)),\mathbf{u}_n(s)) \ \mathrm{d}s \geq -(T-(s+t))M_\mathbf{L}\mathbf{1}$$ and $\mathbf{d} \in  \mathbf{R}^p_+,$
we finally obtain $\mathbf{z'}(s) \geq
-(T-t(s))M_\mathbf{L}\mathbf{1}.$ Hence,
$(t(s),\mathbf{x'}(s),\mathbf{z'}(s))\in \mathcal{H}$ and
$(t(\cdot),\mathbf{x'}(\cdot),\mathbf{z'}(\cdot))$ is viable in
$\mathcal{H}.$\\
\end{proof}

\begin{remark} \label{rem:biggerM}
Proposition \emph{\ref{prop:char}} remains valid if we take for
$\mathcal{H}$ the closed set $\{ (t,\mathbf{x},\mathbf{z}), \ t \in
[0,T], \ \mathbf{x} \in \mathbf{R}^n, \mathbf{z} \in
\{-(T-t)\overline{M}_\mathbf{L} \mathbf{1}\}+ \mathbf{R}^p_+$, where
$\overline{M}_\mathbf{L} > M_\mathbf{L}.$ This remark will be used
in \emph{\S \ref{s:conv}}.
\end{remark}

\section{Approximation of $\mathrm{Viab}_\phi(\mathcal{H})$} \label{s:approx}

In this section, we approximate $\mathrm{Viab}_\phi(\mathcal{H})$ by finite discrete viability kernels. A preliminary step is to approximate $\mathrm{Viab}_\phi(\mathcal{H})$ by discrete viability kernels. Our developments  closely follow  \cite{Cardaliaguet99_Via}.

\subsection{Approximation of $\mathrm{Viab}_\phi(\mathcal{H})$ by discrete viability kernels}

In \cite{Cardaliaguet99_Via} (Theorem 2.14, p.\ 190), it is shown that $\mathrm{Viab}_\phi(\mathcal{H})$
can be approximated by discrete viability kernels in the sense of Painlev\'e-Kuratowski by considering an approximation $\phi_\epsilon$ of $\phi$ satisfying the  following  three properties:\\
\begin{description}
  \item[(\emph{$\mathbf{H_0}$})] $\phi_\epsilon$ is an upper semicontinuous set-valued map from $\mathbf{R} \times \mathbf{R}^n  \times \mathbf{R}^p$ to $\mathbf{R} \times \mathbf{R}^n  \times \mathbf{R}^p$ which takes
    convex compact nonempty values.
  \item[(\emph{$\mathbf{H_1}$})] $$\mathrm{Graph}(\phi_\epsilon) \subset \mathrm{Graph}(\phi) + g(\epsilon)\mathbf{B} \ \mathrm{where} \ \lim_{\epsilon \rightarrow 0^+} g(\epsilon) = 0^+.$$
  \item[(\emph{$\mathbf{H_2}$})] $$\forall (t_\epsilon,\mathbf{x}_\epsilon,\mathbf{z}_\epsilon) \in \mathbf{R} \times \mathbf{R}^n \times \mathbf{R}^p, \ \bigcup_{\| (t,\mathbf{x},\mathbf{z}) - (t_\epsilon,\mathbf{x}_\epsilon,\mathbf{z}_\epsilon)\| \leq M \epsilon} \phi(t,\mathbf{x},\mathbf{z}) \subset \phi_\epsilon(t_\epsilon,\mathbf{x}_\epsilon,\mathbf{z}_\epsilon),$$
\end{description}
where $M = \max\{1,M_{\mathrm{FL}}\}$ denotes a bound for $\phi$ and $\epsilon > 0$ is the time step discretization.\\

Define the set-valued map $\phi_\epsilon, \ \epsilon > 0,$ from $\mathbf{R} \times \mathbf{R}^n  \times \mathbf{R}^p$
to $\mathbf{R} \times\mathbf{R}^n  \times \mathbf{R}^p$ by

\begin{equation} \label{eq:expandedynamicstimediscre}
\phi_\epsilon(t,\mathbf{x},\mathbf{z}) = \left \{ \begin{array}{lll} \{1\} \times (\mathrm{FL^-}(\mathbf{x}) + \epsilon KM \mathbf{B})& \mathrm{if}  & t < T - M \epsilon,\\ {[}0,1{]} \times \overline{\mathrm{co}}((\mathrm{FL^-}(\mathbf{x})+ \epsilon KM \mathbf{B}) \cup \{(\mathbf{0},\mathbf{0})\}) & \mathrm{if}  & t \geq T - M \epsilon.\end{array}\right.
\end{equation} \\

\begin{theorem}
The set-valued map $\phi_\epsilon$ satisfies \emph{($\mathbf{H_0}$)}, \emph{($\mathbf{H_1}$)}, and \emph{($\mathbf{H_2}$)}.
\end{theorem}
\begin{proof}
\begin{description}
  \item[(\emph{$\mathbf{H_0}$})] This follows  from the properties of the set-valued map $\mathrm{FL^-}$
  and the fact that $\mathrm{FL^-}(\mathbf{x}) + \epsilon KM \mathbf{B} \subset \overline{\mathrm{co}}((\mathrm{FL^-}(\mathbf{x})+ \epsilon KM \mathbf{B}) \cup \{(\mathbf{0},\mathbf{0})\})$.\\
  \item[(\emph{$\mathbf{H_1}$})] This relation holds with $g(\epsilon) =  \epsilon \max \{1,K\}M.$ \\ \\
Let $(t_\epsilon,\mathbf{x}_\epsilon,\mathbf{z}_\epsilon) \in \mathbf{R} \times \mathbf{R}^n \times \mathbf{R}^p$ and $(s_\epsilon,\mathbf{f}_\epsilon,\mathbf{l}_\epsilon) \in \phi_\epsilon(t_\epsilon,\mathbf{x}_\epsilon,\mathbf{z}_\epsilon).$ Consider the following two cases: \\
\begin{itemize}
  \item $t_\epsilon < T - M \epsilon.$ Then, $s_\epsilon = 1$ and $(\mathbf{f}_\epsilon,\mathbf{l}_\epsilon)
\in \mathrm{FL^-}(\mathbf{x}_\epsilon) + \epsilon KM \mathbf{B}.$ Hence, $g(\epsilon) =  \epsilon KM.$
  \item  $t_\epsilon \geq T - M \epsilon.$ We have $s_\epsilon \in [0,1]$ and $(\mathbf{f}_\epsilon,\mathbf{l}_\epsilon) \in \overline{\mathrm{co}}((\mathrm{FL^-}(\mathbf{x_\epsilon})+ \epsilon KM \mathbf{B}) \cup \{(\mathbf{0},\mathbf{0})\}).$ There exists $t \geq T$ such that $|t-t_\epsilon| \leq M \epsilon.$ We have
   \begin{equation} \nonumber
 \begin{array}{rcl} (\mathbf{f}_\epsilon,\mathbf{l}_\epsilon)  & \in & \overline{\mathrm{co}}((\mathrm{FL^-}(\mathbf{x}_\epsilon) +  \epsilon KM\mathbf{B}) \cup \{(\mathbf{0},\mathbf{0})\}),\\
& \subset  & \overline{\mathrm{co}}(\mathrm{FL^-}(\mathbf{x}_\epsilon ) \cup \{(\mathbf{0},\mathbf{0})\}) +  \epsilon KM\mathbf{B}.
 \end{array}
\end{equation}
Hence, $g(\epsilon) =  \epsilon \max \{1,K\}M.$\\

\end{itemize}

  \item[(\emph{$\mathbf{H_2}$})]
      Let $(t_\epsilon,\mathbf{x}_\epsilon,\mathbf{z}_\epsilon) \in \mathbf{R} \times \mathbf{R}^n \times \mathbf{R}^p$. Take $(t,\mathbf{x},\mathbf{z}) \in
      \mathbf{R} \times \mathbf{R}^n \times \mathbf{R}^p$ such that $\| (t,\mathbf{x},\mathbf{z}) - (t_\epsilon,\mathbf{x}_\epsilon,\mathbf{z}_\epsilon)\| \leq M \epsilon$. In particular, $\|\mathbf{x}-\mathbf{x}_\epsilon\| \leq M \epsilon$. Consider the two following cases:\\
\begin{itemize}
  \item $t_\epsilon < T - M\epsilon.$  This implies $t < T$. The Lipschitz property of $\mathrm{FL^-}$
yields:
$$\phi(t,\mathbf{x},\mathbf{z}) = \{1\} \times \mathrm{FL^-}(\mathbf{x}) \subset \{1\} \times (\mathrm{FL^-}(\mathbf{x}_\epsilon) + \epsilon KM \mathbf{B}) = \phi_\epsilon(t_\epsilon,\mathbf{x}_\epsilon,\mathbf{z}_\epsilon).$$

  \item $t_\epsilon \geq T - M\epsilon.$ Then, either $t < T.$ In such a case, as above,
$$\phi(t,\mathbf{x},\mathbf{z}) = \{1\} \times \mathrm{FL^-}(\mathbf{x}) \subset \{1\} \times (\mathrm{FL^-}(\mathbf{x}_\epsilon) + \epsilon KM \mathbf{B}) \subset \phi_\epsilon(t_\epsilon,\mathbf{x}_\epsilon,\mathbf{z}_\epsilon).$$
Either $t \geq T.$ In such a case, using the Lipschitz property of $\mathrm{FL^-}$, we have
$$\mathrm{FL^-}(\mathbf{x}) \cup \{(\mathbf{0},\mathbf{0})\} \subset (\mathrm{FL^-}(\mathbf{x}_\epsilon) + \epsilon KM \mathbf{B}) \cup \{(\mathbf{0},\mathbf{0})\}.$$
Hence,
$$ \phi(t,\mathbf{x},\mathbf{z}) = [0,1] \times \overline{\mathrm{co}}(\mathrm{FL^-}(\mathbf{x}) \cup \{(\mathbf{0},\mathbf{0})\}) \subset \phi_\epsilon(t_\epsilon,\mathbf{x}_\epsilon,\mathbf{z}_\epsilon).$$
\end{itemize}
\end{description}
\end{proof}

\subsection{Approximation of $\mathrm{Viab}_\phi(\mathcal{H})$ by discrete finite viability kernels}

In \cite{Cardaliaguet99_Via} (Theorem 2.19, p.\ 195), it is shown that $\mathrm{Viab}_\phi(\mathcal{H})$
can be approximated by finite discrete viability kernels in the sense of Painlev\'e-Kuratowski by considering an approximation $\Gamma_{\epsilon,h}$ of $G_\epsilon(\cdot)$ satisfying the  following two properties:\\
\begin{description}
  \item[(\emph{$\mathbf{H_3}$})] $$\mathrm{Graph}(\Gamma_{\epsilon,h}) \subset \mathrm{Graph}(G_\epsilon) + \psi(\epsilon,h)\mathbf{B} \ \mathrm{where} \ \lim_{\epsilon \rightarrow 0^+,\frac{h}{\epsilon}\rightarrow 0^+} \frac{\psi(\epsilon,h)}{\epsilon} = 0^+.$$
  \item[(\emph{$\mathbf{H_4}$})] $\forall (t_h,\mathbf{x_h},\mathbf{z_h}) \in \mathbf{R}_h \times \mathbf{R}^n_h \times \mathbf{R}^p_h,$ $$\bigcup_{\| (t_\epsilon,\mathbf{x}_\epsilon,\mathbf{z}_\epsilon) - (t_h,\mathbf{x_h},\mathbf{z_h})\| \leq h} (G_\epsilon(t_\epsilon,\mathbf{x}_\epsilon,\mathbf{z}_\epsilon)+h\mathbf{B})\cap \mathbf{R}_h \times \mathbf{R}^n_h \times \mathbf{R}^p_h \subset \Gamma_{\epsilon,h}(t_h,\mathbf{x_h},\mathbf{z_h}),$$
\end{description}
where $ h>0$ is the state step discretization, $\mathbf{R}_h$ is an
integer lattice of $\mathbf{R}$ generated by segments of length $h$,
$G_\epsilon$ is the set-valued map from $\mathbf{R} \times
\mathbf{R}^n \times \mathbf{R}^p$ to $\mathbf{R} \times \mathbf{R}^n
\times \mathbf{R}^p$ defined by
$$G_\epsilon(t_\epsilon,\mathbf{x}_\epsilon,\mathbf{z}_\epsilon) = \{(t_\epsilon,\mathbf{x}_\epsilon,\mathbf{z}_\epsilon)\} +  \epsilon \phi_\epsilon(t_\epsilon,\mathbf{x}_\epsilon,\mathbf{z}_\epsilon),$$
and $\Gamma_{\epsilon,h}$ is the set-valued map from $\mathbf{R}_h \times \mathbf{R}^n_h  \times \mathbf{R}^p_h$ to $\mathbf{R}_h \times \mathbf{R}^n_h  \times \mathbf{R}^p_h$
defined as follows.\\

\begin{itemize}
  \item $\mathrm{If} \  t_h < T - M \epsilon - h$, $ \Gamma_{\epsilon,h}(t_h,\mathbf{x_h},\mathbf{z_h}) = $
  \begin{equation} \label{eq:expandedynamicsstatediscre1}  [t_h+\epsilon-2h,t_h+\epsilon+2h] \cap \mathbf{R}_h \times (\{(\mathbf{x_h},\mathbf{z_h})\} + \epsilon
\mathrm{FL^-}(\mathbf{x_h}) + \alpha_{\epsilon,h}\mathbf{B}) \cap
\mathbf{R}^n_h  \times \mathbf{R}^p_h.\end{equation}
  \item $\mathrm{If} \  t_h \geq T - M \epsilon - h$, $ \Gamma_{\epsilon,h}(t_h,\mathbf{x_h},\mathbf{z_h}) = $
  \begin{equation}  \label{eq:expandedynamicsstatediscre2} {[}t_h,t_h+\epsilon+2h{]} \cap \mathbf{R}_h \times \overline{\mathrm{co}}((\{(\mathbf{x_h},\mathbf{z_h})\} + \epsilon
\mathrm{FL^-}(\mathbf{x_h}) + \alpha_{\epsilon,h}\mathbf{B}) \cup
(\{(\mathbf{x_h},\mathbf{z_h})\}+2h\mathbf{B})) \cap \mathbf{R}^n_h
\times \mathbf{R}^p_h,\end{equation}
where $ \alpha_{\epsilon,h} = 2h+\epsilon hK+\epsilon^2 KM $.\\
\end{itemize}
We assume that $\epsilon > 2h$.\\


\begin{theorem}
The set-valued map $\Gamma_{\epsilon,h}$ satisfies \emph{($\mathbf{H_3}$)} and \emph{($\mathbf{H_4}$)}.
\end{theorem}

\begin{proof}
\begin{description}
  \item[(\emph{$\mathbf{H_3}$})]
This relation holds with $\psi(\epsilon,h) = 4h+\epsilon hK,$ which verifies $$\lim_{\epsilon \rightarrow 0^+,\frac{h}{\epsilon}\rightarrow 0^+} \frac{(4+\epsilon K)h}{\epsilon}  = 0^+.$$ \\ \\
Let $(t_h,\mathbf{x_h},\mathbf{z_h}) \in \mathbf{R}_h \times \mathbf{R}^n_h \times \mathbf{R}^p_h$ and $(s_h,\mathbf{f_h},\mathbf{l_h}) \in \Gamma_{\epsilon,h}(t_h,\mathbf{x_h},\mathbf{z_h}).$ Consider the following two cases: \\
\begin{itemize}
  \item $t_h < T - M \epsilon - h.$ Then, $s_h \in [t_h+\epsilon-2h,t_h+\epsilon+2h] \cap \mathbf{R}_h \subset t_h + \epsilon + [-2h,2h]$, and
\begin{equation} \nonumber
 \begin{array}{rcl}
(\mathbf{f_h},\mathbf{l_h}) & \in & (\{(\mathbf{x_h},\mathbf{z_h})\}
+ \epsilon \mathrm{FL^-}(\mathbf{x_h}) +
\alpha_{\epsilon,h}\mathbf{B}) \cap \mathbf{R}^n_h  \times
\mathbf{R}^p_h, \\ & \subset & \{(\mathbf{x_h},\mathbf{z_h})\} +
\epsilon \mathrm{FL^-}(\mathbf{x_h}) +
\alpha_{\epsilon,h}\mathbf{B}, \\ &  = &
\{(\mathbf{x_h},\mathbf{z_h})\} + \epsilon (
\mathrm{FL^-}(\mathbf{x_h}) + \epsilon KM\mathbf{B}) + (2h+\epsilon
hK)\mathbf{B}.
 \end{array}
\end{equation}
Hence, $(s_h,\mathbf{f_h},\mathbf{l_h}) \in
(t_h,\mathbf{x_h},\mathbf{z_h}) +
\epsilon\phi_\epsilon(t_h,\mathbf{x_h},\mathbf{z_h}) +
\max\{2h,2h+\epsilon hK\}\mathbf{B} =
G_\epsilon(t_h,\mathbf{x_h},\mathbf{z_h}) + (2h+\epsilon
hK)\mathbf{B}$.
  \item  $t_h \geq T - M \epsilon - h.$ We have $s_h \in {[}t_h,t_h+\epsilon+2h{]} \cap \mathbf{R}_h$ and
$(\mathbf{f_h},\mathbf{l_h}) \in  \overline{\mathrm{co}}((\{(\mathbf{x_h},\mathbf{z_h})\} + \epsilon
\mathrm{FL^-}(\mathbf{x_h}) + \alpha_{\epsilon,h}\mathbf{B}) \cup (\{(\mathbf{x_h},\mathbf{z_h})\}+h\mathbf{B})) \cap \mathbf{R}^n_h  \times \mathbf{R}^p_h.$ There exists $t_\epsilon \geq T-M \epsilon $ such that $ 0 \leq t_\epsilon-t_h \leq h.$ We have
   \begin{equation} \nonumber
 \begin{array}{rcl} s_h  & \in & {[}t_h,t_h+\epsilon+2h{]} \cap \mathbf{R}_h,\\
& \subset  & {[}t_h,t_h+\epsilon+2h{]},\\
& \subset  & t_\epsilon+ \epsilon[0,1] + [-2h,2h].\\
 \end{array}
\end{equation}
Moreover,
\begin{equation} \nonumber
 \begin{array}{rcl} (\mathbf{f_h},\mathbf{l_h})  & \in & \overline{\mathrm{co}}((\{(\mathbf{x_h},\mathbf{z_h})\} + \epsilon
\mathrm{FL^-}(\mathbf{x_h}) + \alpha_{\epsilon,h}\mathbf{B}) \cup (\{(\mathbf{x_h},\mathbf{z_h})\}+2h\mathbf{B})) \cap \mathbf{R}^n_h  \times \mathbf{R}^p_h,\\  & \subset  & \overline{\mathrm{co}}((\{(\mathbf{x_h},\mathbf{z_h})\} + \epsilon
\mathrm{FL^-}(\mathbf{x_h}) + \alpha_{\epsilon,h}\mathbf{B}) \cup (\{(\mathbf{x_h},\mathbf{z_h})\}+2h\mathbf{B})),\\
& =  & \{(\mathbf{x_h},\mathbf{z_h})\} + \overline{\mathrm{co}}((\epsilon
\mathrm{FL^-}(\mathbf{x_h}) + \alpha_{\epsilon,h}\mathbf{B}) \cup (\{(\mathbf{0},\mathbf{0})\}+2h\mathbf{B})),\\
& =  & \{(\mathbf{x_h},\mathbf{z_h})\} + \overline{\mathrm{co}}((\epsilon(\mathrm{FL^-}(\mathbf{x_h}) + \epsilon KM\mathbf{B})+(2h+\epsilon hK)\mathbf{B}) \cup (\{(\mathbf{0},\mathbf{0})\}+2h\mathbf{B})),\\
& \subset  & \{(\mathbf{x_h},\mathbf{z_h})\} + \overline{\mathrm{co}}((\epsilon(\mathrm{FL^-}(\mathbf{x_h}) + \epsilon KM\mathbf{B})) \cup \{(\mathbf{0},\mathbf{0})\})+(4h+\epsilon hK)\mathbf{B}.\\
 \end{array}
\end{equation}
Hence, $(s_h,\mathbf{f_h},\mathbf{l_h}) \in (t_\epsilon,\mathbf{x_h},\mathbf{z_h}) + \epsilon\phi_\epsilon(t_\epsilon,\mathbf{x_h},\mathbf{z_h}) + \max\{4h+\epsilon hK,2h\}\mathbf{B} = G_\epsilon(t_\epsilon,\mathbf{x_h},\mathbf{z_h}) + (4h+\epsilon hK)\mathbf{B}$.\\
\end{itemize}
  \item[(\emph{$\mathbf{H_4}$})]
      Let $(t_h,\mathbf{x_h},\mathbf{z_h}) \in \mathbf{R}_h \times \mathbf{R}^n_h \times \mathbf{R}^p_h$.
Take $(t_\epsilon,\mathbf{x}_\epsilon,\mathbf{z}_\epsilon) \in
      \mathbf{R} \times \mathbf{R}^n \times \mathbf{R}^p$ such that
      $\| (t_\epsilon,\mathbf{x}_\epsilon,\mathbf{z}_\epsilon) - (t_h,\mathbf{x_h},\mathbf{z_h})\| \leq h$. In particular,  $|t_\epsilon-t_h| \leq
      h$ and $\|\mathbf{x}-\mathbf{x_h}\| \leq h$. Consider the two following cases:\\
\begin{itemize}
  \item $t_h < T - M \epsilon - h.$  This implies $t_\epsilon < T - M \epsilon$. The Lipschitz property of $\mathrm{FL^-}$
yields:
\begin{equation} \nonumber
 \begin{array}{rcl} G_\epsilon(t_\epsilon,\mathbf{x}_\epsilon,\mathbf{z}_\epsilon)  + h\mathbf{B} & = & (t_\epsilon,\mathbf{x}_\epsilon,\mathbf{z}_\epsilon) + \epsilon \phi_\epsilon(t_\epsilon,\mathbf{x}_\epsilon,\mathbf{z}_\epsilon) + h\mathbf{B}, \\  & =  & (t_\epsilon + \epsilon + [-h,h]) \times
(\{(\mathbf{x}_\epsilon,\mathbf{z}_\epsilon)\} + \epsilon \mathrm{FL^-}(\mathbf{x}_\epsilon) +\epsilon^2 KM\mathbf{B}  + h\mathbf{B}), \\
& \subset  & (t_h + \epsilon + [-2h,2h]) \times
(\{(\mathbf{x}_h,\mathbf{z}_h)\} + h\mathbf{B} + \epsilon \mathrm{FL^-}(\mathbf{x}_h) + \epsilon h K\mathbf{B} + \epsilon^2 KM+h\mathbf{B}), \\
& =  & (t_h + \epsilon + [-2h,2h]) \times
(\{(\mathbf{x}_h,\mathbf{z}_h)\} + \epsilon \mathrm{FL^-}(\mathbf{x}_h) + \alpha_{\epsilon,h}\mathbf{B}). \\
 \end{array}
\end{equation}
Hence, $$
(G_\epsilon(t_\epsilon,\mathbf{x}_\epsilon,\mathbf{z}_\epsilon)+h\mathbf{B})\cap
\mathbf{R}_h \times \mathbf{R}^n_h \times \mathbf{R}^p_h\subset
\Gamma_{\epsilon,h}(t_h,\mathbf{x_h},\mathbf{z_h}).$$
  \item $t_h \geq T - M \epsilon - h.$ Then, either  $t_\epsilon < T-M\epsilon.$ In such a case, as above,
$$G_\epsilon(t_\epsilon,\mathbf{x}_\epsilon,\mathbf{z}_\epsilon)  + h\mathbf{B} \subset (t_h + \epsilon + [-2h,2h]) \times
(\{(\mathbf{x}_h,\mathbf{z}_h)\} + \epsilon
\mathrm{FL^-}(\mathbf{x}_h) + \alpha_{\epsilon,h}\mathbf{B}).$$ As
$\epsilon > 2h$, $$t_h + \epsilon + [-2h,2h]  \subset
{[}t_h,t_h+\epsilon+2h{]}.$$ Moreover,
$$\{(\mathbf{x}_h,\mathbf{z}_h)\} + \epsilon
\mathrm{FL^-}(\mathbf{x}_h) + \alpha_{\epsilon,h}\mathbf{B} \subset
 \overline{\mathrm{co}}((\{(\mathbf{x_h},\mathbf{z_h})\} + \epsilon
\mathrm{FL^-}(\mathbf{x_h}) + \alpha_{\epsilon,h}\mathbf{B}) \cup
(\{(\mathbf{x_h},\mathbf{z_h})\}+2h\mathbf{B})). $$
 Hence, $$
(G_\epsilon(t_\epsilon,\mathbf{x}_\epsilon,\mathbf{z}_\epsilon)+h\mathbf{B})\cap
\mathbf{R}_h \times \mathbf{R}^n_h \times \mathbf{R}^p_h\subset
\Gamma_{\epsilon,h}(t_h,\mathbf{x_h},\mathbf{z_h}).$$ Either,
$t_\epsilon \geq T-M\epsilon.$ In which case, using the  Lipschitz
property of $\mathrm{FL^-}$, we have
   \begin{equation} \nonumber
 \begin{array}{rcl} G_\epsilon(t_\epsilon,\mathbf{x}_\epsilon,\mathbf{z}_\epsilon)  + h\mathbf{B} & = & (t_\epsilon,\mathbf{x}_\epsilon,\mathbf{z}_\epsilon)
 + \epsilon ({[}0,1{]} \times \overline{\mathrm{co}}((\mathrm{FL^-}(\mathbf{x}_\epsilon)+ \epsilon KM \mathbf{B}) \cup \{(\mathbf{0},\mathbf{0})\})) + h\mathbf{B}, \\
& =  &   ([t_\epsilon-h,t_\epsilon + \epsilon + h]) \times  (\overline{\mathrm{co}}(\{(\mathbf{x}_\epsilon,\mathbf{z}_\epsilon) + \epsilon \mathrm{FL^-}(\mathbf{x}_\epsilon) +  \epsilon^2 KM\mathbf{B}\} \cup \{(\mathbf{x}_\epsilon,\mathbf{z}_\epsilon)\}) + h\mathbf{B}).\\
\end{array}
\end{equation}
Now,
   \begin{equation} \nonumber
 \begin{array}{l} \overline{\mathrm{co}}(\{(\mathbf{x}_\epsilon,\mathbf{z}_\epsilon) + \epsilon \mathrm{FL^-}(\mathbf{x}_\epsilon) +  \epsilon^2 KM\mathbf{B}\} \cup \{(\mathbf{x}_\epsilon,\mathbf{z}_\epsilon)\}) + h\mathbf{B},  \\ \subset \overline{\mathrm{co}}(\{(\mathbf{x}_h,\mathbf{z}_h) + h\mathbf{B} + \epsilon \mathrm{FL^-}(\mathbf{x}_h) +  \epsilon Kh + \epsilon^2 KM\mathbf{B}\} \cup (\{(\mathbf{x}_h,\mathbf{z}_h)\}+h\mathbf{B})) + h\mathbf{B}, \\
\subset \overline{\mathrm{co}}(\{(\mathbf{x}_h,\mathbf{z}_h) + \epsilon \mathrm{FL^-}(\mathbf{x}_h) + \alpha_{\epsilon,h}\mathbf{B}\} \cup (\{(\mathbf{x}_h,\mathbf{z}_h)\}+2h\mathbf{B})).\\
\end{array}
\end{equation}
Hence, $$
(G_\epsilon(t_\epsilon,\mathbf{x}_\epsilon,\mathbf{z}_\epsilon)+h\mathbf{B})\cap
\mathbf{R}_h \times \mathbf{R}^n_h \times \mathbf{R}^p_h\subset
\Gamma_{\epsilon,h}(t_h,\mathbf{x_h},\mathbf{z_h}).$$
\end{itemize}
\end{description}
\end{proof}

\section{Convergent approximation of the set-valued return function
$V$} \label{s:conv} In this section, we first introduce a sequence
of approximate set-valued return functions with finite set-values
recursively defined by a multiobjective dynamic programming equation
\cite{Guigue11,Guigue09}. We then show that the epigraphs of these
approximate set-valued return functions are equal to the sets
involved in the calculation of the finite discrete viability kernels
of  the discrete set $\mathcal{H}_h = (\mathcal{H} + h\mathbf{B})
\cap \mathbf{R}_{h} \times \mathbf{R}_{h}^n \times \mathbf{R}_{h}^p$
for the finite discrete dynamics $\Gamma_{\epsilon,h}$ (Proposition
2.18, p.\ 195, \cite{Cardaliaguet99_Via}). This allows us to conclude
that the sequence of approximate set-valued return functions is
finite and that the epigraph of the final approximate set-valued
return function of this sequence converges in the sense of
Painlev\'e-Kuratowski towards the epigraph of the set-valued
return function $V$.\\

Recall the definition of $\mathcal{H} = \{
(t,\mathbf{x},\mathbf{z}), \ t \in [0,T], \ \mathbf{x} \in
\mathbf{R}^n, \mathbf{z} \in \{-(T-t)M_\mathbf{L} \mathbf{1}\}+
\mathbf{R}^p_+$. Here, we take $\overline{M}_\mathbf{L} >
M_\mathbf{L}$ (Remark \ref{rem:biggerM}) in the definition of
$\mathcal{H}$. Hence, $\mathcal{H} = \{ (t,\mathbf{x},\mathbf{z}), \
t \in [0,T], \ \mathbf{x} \in \mathbf{R}^n, \mathbf{z} \in \{-(T-t)
\overline{M}_\mathbf{L} \mathbf{1}\}+ \mathbf{R}^p_+$. Let $I_h =
(I+[-h,h])\cap \mathbf{R}_{h}$. We define the finite-valued
set-valued map $V^0_{\epsilon,h}$ from $I_h \times \mathbf{R}_{h}^n$
to $\mathbf{R}_{h}^p$ such that
$$\mathrm{Graph} (V^0_{\epsilon,h} + \mathbf{R}_{h,+}^p) =
\mathcal{H}_h,$$ where $\mathcal{H}_h = (\mathcal{H} + h\mathbf{B})
\cap \mathbf{R}_{h} \times \mathbf{R}_{h}^n \times
\mathbf{R}_{h}^p.$ We recursively define now the finite set-valued
maps $V^k_{\epsilon,h}, \ k \geq 1,$ from $I_h \times
\mathbf{R}_{h}^n$ to $\mathbf{R}_{h}^p$ as
follows:\\
\begin{itemize}
    \item $\mathrm{If} \ t_h < T - M \epsilon - h,$
    $V^{k+1}_{\epsilon,h}(t_h,\mathbf{x_h}) = $
    \begin{multline} \label{eq:recursive_1}
      \mathcal{E}\bigg(\bigg\{(\epsilon\mathbf{l}+ \alpha_{\epsilon,h} \mathbf{B})\cap \mathbf{R}^p_h+ V^{k}_{\epsilon,h}((t_h+\epsilon+[-2h,2h])\cap \mathbf{R}_h,(\mathbf{x_h}+\epsilon \mathbf{f} +\alpha_{\epsilon,h} \mathbf{B}) \cap \mathbf{R}^n_h), \ (\mathbf{f},\mathbf{l}) \in \mathrm{FL^+}(\mathbf{x_h})\bigg
      \}\bigg).
    \end{multline}
    \item $\mathrm{Otherwise}$,
    \begin{equation} \label{eq:recursive_2}
    V^{k+1}_{\epsilon,h}(t_h,\mathbf{x_h}) = V^k_{\epsilon,h}(t_h,\mathbf{x_h}).
    \end{equation}
\end{itemize}
\begin{remark}
The closure in \emph{(\ref{eq:recursive_1})} is not required anymore
as the sets involved are finite.
\end{remark}\\

We aim in the following two propositions to prove that
$$ \mathrm{Graph}(V^{k+1}_{\epsilon,h} + \mathbf{R}_{h,+}^p) \subset \mathrm{Graph}(V^{k}_{\epsilon,h} + \mathbf{R}_{h,+}^p),$$
and
$$ \mathrm{Graph}(V^{k}_{\epsilon,h} + \mathbf{R}_{h,+}^p) = A^k,$$
where the sets $A^k$ are recursively defined (Proposition 2.18, p.\ 195,
\cite{Cardaliaguet99_Via}) from $A^0 = \mathcal{H}_h$ and the
relation
$$ A^{k+1} = \{(t_h,\mathbf{x_h},\mathbf{z_h}) \in A^{k} \
\mathrm{s.t.} \ \Gamma_{\epsilon,h}(t_h,\mathbf{x_h},\mathbf{z_h})
\cap A^{k} \neq \emptyset \}.$$\\
To simplify the proof of Proposition \ref{prop:decreasingraph}, we
will assume that $T$ is a multiple of $h$ and that  $\epsilon - 2h >
2h$, which guarantees that $\forall t_h \in I_h, \ t_h + \epsilon
-2h > h$. We will also assume that for all $(t_h,\mathbf{x_h}) \in
I_h \times \mathbf{R}^n_h$ such that $t_h \geq h,$
$$V^0_{\epsilon,h}(t_h,\mathbf{x_h}) = \{(-(T+h-t_h)\overline{M}_\mathbf{L}-h)\mathbf{1}\}.$$
This guarantees that, for all $(t_h,\mathbf{x_h}) \in I_h \times
\mathbf{R}^n_h$,  for all $(\mathbf{f},\mathbf{l}) \in
\mathrm{FL^+}(\mathbf{x_h})$, for all $\widetilde{t}_h \in
(t_h+\epsilon+[-2h,2h])\cap \mathbf{R}_h$, for all
$\mathbf{\widetilde{x}_h} \in (\mathbf{x_h}+\epsilon \mathbf{f}
+\alpha_{\epsilon,h}) \cap \mathbf{R}_h^n$,
$$V^0_{\epsilon,h}(\widetilde{t}_h,\mathbf{\widetilde{x}_h}) = \{(-(T+h-\widetilde{t}_h)\overline{M}_\mathbf{L}-h)\mathbf{1}\}.$$
We will finally assume that $\epsilon,h,$ and
$\overline{M}_\mathbf{L}$ have been chosen such that
\begin{equation} \label{eq:keyrelation}\epsilon M_\mathbf{L} +
\alpha_{\epsilon,h} \leq
(\epsilon-2h)\overline{M}_\mathbf{L}.\end{equation}

\begin{proposition} \label{prop:decreasingraph}
$$\forall k, \ \mathrm{Graph}(V^{k+1}_{\epsilon,h} + \mathbf{R}_{h,+}^p) \subset \mathrm{Graph}(V^{k}_{\epsilon,h} + \mathbf{R}_{h,+}^p).$$
\end{proposition}
\begin{proof}
We start by proving that this relation holds for $k=0$. Take
$(t_h,\mathbf{x_h},\mathbf{z_h}) \in
\mathrm{Graph}(V^{1}_{\epsilon,h} + \mathbf{R}_{h,+}^p)$. Hence,
$\mathbf{z_h} \in V^{1}_{\epsilon,h}(t_h,\mathbf{x_h})
+ \mathbf{R}_{h,+}^p$. We have two cases to consider:\\
\begin{enumerate}
    \item   $\mathrm{If} \ t_h \geq T - M \epsilon - h.$ From (\ref{eq:recursive_2}), we directly get
    $\mathbf{z_h} \in V^{0}_{\epsilon,h}(t_h,\mathbf{x_h}) +  \mathbf{R}_{h,+}^p$.
    \item Otherwise, $t_h < T - M \epsilon - h.$  By (\ref{eq:recursive_1}),
        \begin{multline} \nonumber
    V^{1}_{\epsilon,h}(t_h,\mathbf{x_h}) = \mathcal{E}\bigg(\bigg\{(\epsilon\mathbf{l}+ \alpha_{\epsilon,h} \mathbf{B})\cap \mathbf{R}^p_h+ V^{0}_{\epsilon,h}((t_h+\epsilon+[-2h,2h])\cap \mathbf{R}_h,(\mathbf{x_h}+\epsilon \mathbf{f} +\alpha_{\epsilon,h} \mathbf{B}\} \cap \mathbf{R}^n_h), \
    (\mathbf{f},\mathbf{l}) \in \mathrm{FL^+}(\mathbf{x_h})\bigg\}\bigg),
    \end{multline}
    Let $\widetilde{t}_h \in (t_h+\epsilon+[-2h,2h])\cap \mathbf{R}_h$,
    $\mathbf{\widetilde{x}_h} \in (\mathbf{x_h}+\epsilon \mathbf{f} +\alpha_{\epsilon,h}
    \mathbf{B}) \cap \mathbf{R}_h^n$, $\mathbf{\overline{z}_h} \in  V^{0}_{\epsilon,h}(\widetilde{t}_h,\mathbf{\widetilde{x}_h})$, and $\mathbf{\widetilde{z}_h}
    \in (\epsilon \mathbf{l} +\alpha_{\epsilon,h}
    \mathbf{B}) \cap \mathbf{R}_h^p$. Hence, $\mathbf{\widetilde{z}_h} \geq -(\epsilon M_\mathbf{L} +\alpha_{\epsilon,h})\mathbf{1}
    $. Moreover, from above,
    we have $$\mathbf{\overline{z}_h} = (-(T+h-\widetilde{t}_h)\overline{M}_\mathbf{L}-h) \mathbf{1}
    \geq  (-(T+h-t_h-(\epsilon-2h))\overline{M}_\mathbf{L}-h) \mathbf{1} =
 (-(T+h-t_h)\overline{M}_\mathbf{L}-h) \mathbf{1} + (\epsilon-2h)\overline{M}_\mathbf{L} \mathbf{1}.$$
    Hence, by (\ref{eq:keyrelation}),  $$\mathbf{\overline{z}_h} \geq (-(T+h-t_h)\overline{M}_\mathbf{L}-h) \mathbf{1} +
    (\epsilon-2h)\overline{M}_\mathbf{L} \mathbf{1} -
    (\epsilon M_\mathbf{L} + \alpha_{\epsilon,h})\mathbf{1} + (\epsilon M_\mathbf{L} + \alpha_{\epsilon,h})\mathbf{1} \geq -(T+h-t_h)\overline{M}_\mathbf{L}\mathbf{1} + (\epsilon M_\mathbf{L} + \alpha_{\epsilon,h}) \mathbf{1}.$$
    Therefore, $$\mathbf{\widetilde{z}_h} + \mathbf{\overline{z}_h} \in V^{0}_{\epsilon,h}(t_h,\mathbf{x_h}) + \mathbf{R}_{h,+}^p.$$
    Hence, $$ V^{1}_{\epsilon,h}(t_h,\mathbf{x_h}) \subset V^{0}_{\epsilon,h}(t_h,\mathbf{x_h}) + \mathbf{R}_{h,+}^p,$$
    from which we get that $$\mathbf{z_h} \in V^{1}_{\epsilon,h}(t_h,\mathbf{x_h}) + \mathbf{R}_{h,+}^p \subset
    V^{0}_{\epsilon,h}(t_h,\mathbf{x_h}) + \mathbf{R}_{h,+}^p + \mathbf{R}_{h,+}^p =  V^{0}_{\epsilon,h}(t_h,\mathbf{x_h}) + \mathbf{R}_{h,+}^p.$$
\end{enumerate}
Assume now that the relation holds up to $k$. We aim to prove that
it holds for $k+1$, i.e.,
$$\mathrm{Graph}(V^{k+2}_{\epsilon,h} + \mathbf{R}_{h,+}^p) \subset \mathrm{Graph}(V^{k+1}_{\epsilon,h} + \mathbf{R}_{h,+}^p).$$
Take  $(t_h,\mathbf{x_h},\mathbf{z_h}) \in
\mathrm{Graph}(V^{k+2}_{\epsilon,h} + \mathbf{R}_{h,+}^p).$
Hence, $\mathbf{z_h} \in V^{k+2}_{\epsilon,h}(t_h,\mathbf{x_h}) + \mathbf{R}_{h,+}^p$. We have two cases to consider:\\
  \begin{enumerate}
    \item  $\mathrm{If} \ t_h \geq T - M \epsilon - h.$ From (\ref{eq:recursive_2}), we directly get
    $\mathbf{z_h} \in V^{k+1}_{\epsilon,h}(t_h,\mathbf{x_h}) +  \mathbf{R}_{h,+}^p$.
    \item Otherwise, $t_h < T - M \epsilon - h.$  By (\ref{eq:recursive_1}),
    for some $(\mathbf{f},\mathbf{l}) \in \mathrm{FL^+}(\mathbf{x_h})$, there exist
    $\widetilde{t}_h \in (t_h+\epsilon+[-2h,2h])\cap \mathbf{R}_h$, $\mathbf{\widetilde{x}_h} \in (\mathbf{x_h}+\epsilon \mathbf{f} +\alpha_{\epsilon,h}
    \mathbf{B}) \cap \mathbf{R}_h^n$, and $\mathbf{\widetilde{z}_h} \in (\epsilon \mathbf{l} +\alpha_{\epsilon,h}
    \mathbf{B}) \cap \mathbf{R}_h^p$ such that
    $$\mathbf{z_h} \in \mathbf{\widetilde{z}_h} + V^{k+1}_{\epsilon,h}(\widetilde{t}_h,\mathbf{\widetilde{x}_h}) +  \mathbf{R}_{h,+}^p.$$
From the induction assumption, we get:
 $$\mathbf{z_h} \in \mathbf{\widetilde{z}_h} + V^{k}_{\epsilon,h}(\widetilde{t}_h,\mathbf{\widetilde{x}_h}) +
\mathbf{R}_{h,+}^p,$$ or,
$$\mathbf{z_h} \in \bigg\{(\epsilon\mathbf{l}+ \alpha_{\epsilon,h} \mathbf{B})\cap \mathbf{R}^p_h+ V^{k}_{\epsilon,h}( (t_h+\epsilon+[-2h,2h])\cap \mathbf{R}_h,
(\mathbf{x_h}+\epsilon \mathbf{f} +\alpha_{\epsilon,h} \mathbf{B})
\cap \mathbf{R}^n_h), \ (\mathbf{f},\mathbf{l}) \in
\mathrm{FL^+}(\mathbf{x_h})\bigg\} +  \mathbf{R}_{h,+}^p.$$ Applying
external stability to $$S=
    \bigg\{(\epsilon\mathbf{l}+ \alpha_{\epsilon,h} \mathbf{B})\cap \mathbf{R}^p_h+ V^{k}_{\epsilon,h}(t_h+\epsilon+[-2h,2h])\cap \mathbf{R}_h,(\mathbf{x_h}+\epsilon
    \mathbf{f} +\alpha_{\epsilon,h} \mathbf{B}) \cap \mathbf{R}^n_h), \ (\mathbf{f},\mathbf{l}) \in \mathrm{FL^+}(\mathbf{x_h})\bigg\},$$ and using (\ref{eq:recursive_1}) yields $\mathbf{z_h} \in V^{k+1}_{\epsilon,h}(t_h,\mathbf{x_h}) +  \mathbf{R}_{h,+}^p$.\newline
\end{enumerate}
\end{proof}

\begin{proposition} \label{prop:chardis}
$$\forall k, \ \mathrm{Graph}(V^{k}_{\epsilon,h} +  \mathbf{R}_{h,+}^p) = A^k.$$
\end{proposition}
\begin{proof}
This relation holds for $k=0$ by definition. Assume now that the
relation holds up to $k$. We aim to prove that it holds for $k+1$,
i.e.,
$$\mathrm{Graph}(V^{k+1}_{\epsilon,h} + \mathbf{R}_{h,+}^p) =
A^{k+1}.$$ First, we prove the inclusion
\begin{equation} \label{first_inclusion}
\mathrm{Graph}(V^{k+1}_{\epsilon,h} + \mathbf{R}_{h,+}^p) \subset
A^{k+1}.\end{equation}
Take $(t_h,\mathbf{x_h}) \in I_h \times \mathbf{R}_h^n$ and $\mathbf{z_h} \in V^{k+1}_{\epsilon,h}(t_h,\mathbf{x_h}) + \mathbf{R}_{h,+}^p$. We have two cases to consider:\\
 \begin{enumerate}
   \item $\mathrm{If} \ t_h \geq T - M \epsilon - h,$ then by (\ref{eq:recursive_2}), we have
  $(t_h,\mathbf{x_h},\mathbf{z_h}) \in \mathrm{Graph}(V^{k}_{\epsilon,h} + \mathbf{R}_{h,+}^p)$.
   Therefore, from the induction assumption, $(t_h,\mathbf{x_h},\mathbf{z_h}) \in A^k$.
    Moreover, by (\ref{eq:expandedynamicsstatediscre2}), we also have $(t_h,\mathbf{x_h},\mathbf{z_h}) \in
\Gamma_{\epsilon,h}(t_h,\mathbf{x_h},\mathbf{z_h})$. Hence,
$(t_h,\mathbf{x_h},\mathbf{z_h}) \in
\Gamma_{\epsilon,h}(t_h,\mathbf{x_h},\mathbf{z_h}) \cap A^k$, or
$\Gamma_{\epsilon,h}(t_h,\mathbf{x_h},\mathbf{z_h}) \cap A^k \neq
\emptyset$, which shows that $(t_h,\mathbf{x_h},\mathbf{z_h}) \in
A^{k+1}$.
   \item Otherwise, $t_h < T - M \epsilon - h.$ By (\ref{eq:recursive_1}), for some
$(\mathbf{f},\mathbf{l}) \in \mathrm{FL^+}(\mathbf{x_h})$, there
exist
   $\widetilde{t}_h \in (t_h+\epsilon+[-2h,2h])\cap \mathbf{R}_h$, $\mathbf{\widetilde{x}_h} \in (\mathbf{x_h}+\epsilon \mathbf{f} +\alpha_{\epsilon,h}
 \mathbf{B}) \cap \mathbf{R}_h^n$, and $\mathbf{\widetilde{z}_h} \in (\epsilon \mathbf{l} +\alpha_{\epsilon,h}
 \mathbf{B}) \cap \mathbf{R}_h^p$ such that $$\mathbf{z_h} \in \mathbf{\widetilde{z}_h}  + V^{k}_{\epsilon,h}(\widetilde{t}_h,\mathbf{\widetilde{x}_h}) +  \mathbf{R}_{h,+}^p.$$ Hence, from the induction assumption, $(\widetilde{t}_h,\mathbf{\widetilde{x}_h},\mathbf{z_h}-\mathbf{\widetilde{z}_h}) \in A^k.$ To get that $(t_h,\mathbf{x_h},\mathbf{z_h}) \in A^{k+1},$
it remains to prove that
   $(t_h,\mathbf{x_h},\mathbf{z_h}) \in A^k$ and $(\widetilde{t}_h,\mathbf{\widetilde{x}_h},\mathbf{z_h}-\mathbf{\widetilde{z}_h}) \in \Gamma_{\epsilon,h}(t_h,\mathbf{x_h},\mathbf{z_h})$ where $\Gamma_{\epsilon,h}(t_h,\mathbf{x_h},\mathbf{z_h})$
    is given by (\ref{eq:expandedynamicsstatediscre1}). $(t_h,\mathbf{x_h},\mathbf{z_h}) \in A^k$ comes from
   Proposition \ref{prop:decreasingraph} and the induction assumption, i.e.,
   $$(t_h,\mathbf{x_h},\mathbf{z_h}) \in \mathrm{Graph}(V^{k+1}_{\epsilon,h} + \mathbf{R}_{h,+}^p) \subset
   \mathrm{Graph}(V^{k}_{\epsilon,h} + \mathbf{R}_{h,+}^p)  = A_k.$$ Moreover, we have:\\
   \begin{enumerate}
\item $\widetilde{t}_h \in (t_h+\epsilon+[-2h,2h])\cap \mathbf{R}_h,$
     \item $\mathbf{\widetilde{x}_h} \in (\mathbf{x_h}+\epsilon \mathbf{f} +\alpha_{\epsilon,h} \mathbf{B}) \cap \mathbf{R}_h^n,$
     \item and \begin{equation} \nonumber
 \begin{array}{rlcl}
&\mathbf{\widetilde{z}_h} & \in & (\epsilon \mathbf{l}
+\alpha_{\epsilon,h}
 \mathbf{B}) \cap \mathbf{R}_h^p\\ \Rightarrow & -\mathbf{\widetilde{z}_h}& \in & (-\epsilon \mathbf{l} +\alpha_{\epsilon,h}
 \mathbf{B}) \cap \mathbf{R}_h^p \\ \Rightarrow & \mathbf{z_h}-\mathbf{\widetilde{z}_h}&  \in & (\mathbf{z_h}-\epsilon \mathbf{l} +\alpha_{\epsilon,h}
 \mathbf{B}) \cap \mathbf{R}_h^p.
 \end{array}
\end{equation}
   \end{enumerate}
 \end{enumerate}
 Hence, $\Gamma_{\epsilon,h}(t_h,\mathbf{x_h},\mathbf{z_h}) \cap A_k \neq \emptyset$  and  (\ref{first_inclusion}) is proved.\\ \\ Conversely, we prove the inclusion  \begin{equation} \label{second_inclusion} A^{k+1} \subset \mathrm{Graph}(V^{k+1}_{\epsilon,h} +  \mathbf{R}_{h,+}^p).\end{equation} Take $(t_h,\mathbf{x_h},\mathbf{z_h}) \in A^{k+1}$.
We have two cases to consider:\\
  \begin{enumerate}
  \item $\mathrm{If} \ t_h \geq T - M \epsilon - h.$  By definition of $A^{k+1}$,
  $(t_h,\mathbf{x_h},\mathbf{z_h}) \in A^{k}.$ Hence,
  from the induction assumption, we get $\mathbf{z_h} \in V^{k}_{\epsilon,h}(t_h,\mathbf{x_h}) +  \mathbf{R}_{h,+}^p$, and
   from (\ref{eq:recursive_2}), $\mathbf{z_h} \in V^{k+1}_{\epsilon,h}(t_h,\mathbf{x_h}) + \mathbf{R}_{h,+}^p.$
  \item Otherwise, $t_h < T - M \epsilon - h.$ Then, there exist
 $(\widetilde{t}_h,\mathbf{\widetilde{x}_h},\mathbf{\widetilde{z}_h}) \in A^k$ such that
$(\widetilde{t}_h,\mathbf{\widetilde{x}_h},\mathbf{\widetilde{z}_h})
\in \ \Gamma_{\epsilon,h}
(t_h,\mathbf{x_h},\mathbf{z_h}),$ or:\\
    \begin{enumerate}
     \item $\widetilde{t}_h \in (t_h + \epsilon +[-2h,2h]) \cap \mathbf{R}_h,$
     \item $\mathbf{\widetilde{x}_h} \in (\mathbf{x_h}+\epsilon \mathbf{f} +\alpha_{\epsilon,h} \mathbf{B}) \cap \mathbf{R}^n_h,$
     \item $\mathbf{\widetilde{z}_h} \in (\mathbf{z_h}-\epsilon \mathbf{l} +\alpha_{\epsilon,h} \mathbf{B}) \cap \mathbf{R}^p_h \Rightarrow   \mathbf{z_h} \in \mathbf{\widetilde{z}_h} +  (\epsilon \mathbf{l} +\alpha_{\epsilon,h} \mathbf{B}) \cap \mathbf{R}^p_h,$\\
    \end{enumerate}
for some $(\mathbf{f},\mathbf{l}) \in \mathrm{FL^-}(\mathbf{x_h}).$
From the induction assumption, we have $\mathbf{\widetilde{z}_h} \in
V^{k}_{\epsilon,h}(\widetilde{t}_h,\mathbf{\widetilde{x}_h}) +
\mathbf{R}_{h,+}^p$. Hence,
$$ \mathbf{z_h} \in   (\epsilon \mathbf{l} +\alpha_{\epsilon,h} \mathbf{B}) \cap \mathbf{R}^p_h + V^{k}_{\epsilon,h}(\widetilde{t}_h,\mathbf{\widetilde{x}_h}) + \mathbf{R}_{h,+}^p.$$
Applying external stability to
    $$S = \bigg\{(\epsilon\mathbf{l}+ \alpha_{\epsilon,h} \mathbf{B})\cap \mathbf{R}^p_h+ V^{k}_{\epsilon,h}(t_h+\epsilon+[-2h,2h])\cap \mathbf{R}_h,(\mathbf{x_h}+\epsilon \mathbf{f} +\alpha_{\epsilon,h} \mathbf{B}) \cap \mathbf{R}^n_h), \ (\mathbf{f},\mathbf{l}) \in \mathrm{FL^+}(\mathbf{x_h})\bigg\},$$ and using (\ref{eq:recursive_1}) yields $\mathbf{z_h} \in V^{k+1}_{\epsilon,h}(t_h,\mathbf{x_h}) +  \mathbf{R}_{h,+}^p$.\\
  \end{enumerate}
\end{proof}

\begin{corollary} The sequence of approximate set-valued return
functions $V^k_{\epsilon,h}$ is finite.
\end{corollary}
\begin{proof}
This follows from Proposition \ref{prop:chardis} and the fact that
the sequence $A_k$ is finite (Proposition
2.18, p.\ 195, \cite{Cardaliaguet99_Via}).\\
\end{proof}

We denote $k(\epsilon,h)$ the last element of this sequence.\\

\begin{corollary} \label{cor:conclusion} The epigraph of the approximate set-valued return function
$V^{k(\epsilon,h)}_{\epsilon,h}$ converges in the sense of
Painlev\'e-Kuratowski towards the epigraph of the set-valued return
function $V$, i.e.,
$$  \mathrm{Graph}(V + \mathbf{R}^p_+) = \lim_{\epsilon \rightarrow 0^+, \ \frac{h}{\epsilon}  \rightarrow 0^+ }  \mathrm{Graph}(V^{k(\epsilon,h)}_{\epsilon,h} +  \mathbf{R}_{h,+}^p).$$
\end{corollary}
\begin{proof}
From \cite{Cardaliaguet99_Via} (Theorem 2.19, p.\ 195), we have
$$   \mathrm{Viab}_\phi(\mathcal{H}) = \lim_{\epsilon \rightarrow 0^+, \ \frac{h}{\epsilon}  \rightarrow 0^+ }   \mathrm{\overrightarrow{Viab}}_{\Gamma_{\epsilon,h}}(\mathcal{H}_h).$$
Moreover, from Proposition \ref{prop:char}, we have
$$  \mathrm{Graph}(V + \mathbf{R}^p_+)  = \mathrm{Viab}_\phi(\mathcal{H}).$$
Finally, from Proposition \ref{prop:chardis} and
\cite{Cardaliaguet99_Via} (Proposition 2.18, p.\ 195), we have
$$\mathrm{Graph}(V^{k(\epsilon,h)}_{\epsilon,h} +  \mathbf{R}_{h,+}^p) = A^{k(\epsilon,h)} =   \mathrm{\overrightarrow{Viab}}_{\Gamma_{\epsilon,h}}(\mathcal{H}_h).$$
The desired result follows.
\end{proof}

\section{A General Numerical Algorithm} \label{s:numalgo}

In this section, we present a general algorithm to determine the
approximate set-valued return function
$V^{k(\epsilon,h)}_{\epsilon,h}$. As shown in Proposition \ref{propalgo}, to find $V^{k(\epsilon,h)}_{\epsilon,h}$, it is not needed to compute $V^{k}_{\epsilon,h}, \ k=0,\ldots,k(\epsilon,h)$ over their entire domain $I_h \times \mathbf{R}_{h}^n$.\\

\begin{proposition} \label{propalgo} $\forall k \geq 0, \ \forall t_h \in I_h, \ t_h \geq T-M\epsilon-h - k (\epsilon-2h), \
\forall \mathbf{x_h} \in \mathbf{R}_{h}^n,$
\begin{equation} \label{eq:approxsvrf}
V^{k+1}_{\epsilon,h}(t_h,\mathbf{x_h}) = V^{k}_{\epsilon,h}(t_h,\mathbf{x_h}).
\end{equation}
\end{proposition}
\begin{proof}
For $k=0$, (\ref{eq:approxsvrf}) directly follows from
(\ref{eq:recursive_2}). Assume now that (\ref{eq:approxsvrf}) holds
for $k > 0$. We prove that (\ref{eq:approxsvrf}) also holds for
$k+1$. Let $t_h \in I_h, \ t_h \geq T-M\epsilon-h - (k+1)
(\epsilon-2h)$ and $\widetilde{t}_h \in (t_h+\epsilon+[-2h,2h]) \cap
\mathbf{R}_{h}$. Then,
$$\widetilde{t}_h \geq t_h+(\epsilon-2h) \geq T-M\epsilon-h - (k+1) (\epsilon-2h) + (\epsilon-2h) = T-M\epsilon-h - k (\epsilon-2h).$$
From the induction assumption, we get $\forall \mathbf{x_h} \in
\mathbf{R}_{h}^n, \ \forall (\mathbf{f},\mathbf{l}) \in
\mathrm{FL^+}(\mathbf{x_h}), \ \forall  \mathbf{\widetilde{x}_h} \in
(\mathbf{x_h}+\epsilon \mathbf{f} +\alpha_{\epsilon,h} B) \cap
\mathbf{R}^n_h,$
$$
V^{k+1}_{\epsilon,h}(\widetilde{t}_h,\mathbf{\widetilde{x}_h}) = V^{k}_{\epsilon,h}(\widetilde{t}_h,\mathbf{\widetilde{x}_h}).
$$
Hence, using (\ref{eq:recursive_1}), we obtain $
V^{k+2}_{\epsilon,h}(t_h,\mathbf{x_h}) = V^{k+1}_{\epsilon,h}(t_h,\mathbf{x_h}),
$ which completes the proof.\\
\end{proof}


We now present  a very general numerical algorithm to approximate
the set $V(0,\mathbf{x_0})$, where for simplicity, we take the
initial state $ \mathbf{x_0}$ in $\mathbf{R}_h^n$. From Corollary
\ref{cor:conclusion}, the suggested approximation to
$V(0,\mathbf{x_0})$ is given by the finite set
$V^{k(\epsilon,h)}_{\epsilon,h}(-h,\mathbf{x_0}).$ The proposed
numerical algorithm is composed of two stages. In the first stage
(\textbf{Algorithm 1}), the computational domain is determined using
the bound on the dynamics. Once the computational domain has been
determined, in the second stage (\textbf{Algorithm
2}), $V^{k(\epsilon,h)}_{\epsilon,h}(-h,\mathbf{x_0})$ is calculated
using the multiple dynamic programming equation (\ref{eq:recursive_1})-(\ref{eq:recursive_2}) together with
Proposition \ref{propalgo}. \\ \\ Choose for example $\epsilon_i =
1/2^{i}$ and $h_i = 1/2^{2i}$. Let $J$ be the number of
discretization steps in $h_i$
for the interval $[-h_i,T+h_i]$, i.e., $J = (T+2h_i)/h_i + 1$ (for simplicity, we assume that $T$ is a multiple of $h_i$) and let $t_j = -h_i + j*h_i$. First, we need to determine the computational domains $\Omega_{j}, \ j=0,\ldots,J-1$.\\ \\
\textbf{Algorithm 1}\\ \\
\textbf{Initialization} $\forall t_j, \  -h_i \leq t_j <
\epsilon_i-3h_i, \ \Omega_{j} = \{\mathbf{x_0}\}.$ Otherwise,
$\Omega_{j} = \emptyset.$\\ \\
\textbf{Main loop}\\
\begin{description}
\item[1]\textbf{Set} $j = 0.$
\item[2]\textbf{Repeat}
    \begin{description}
    \item[2.1] \textbf{Set} $\mathbf{x_h}$ to the first grid point in $\Omega_{j}.$
    \item[2.2] \textbf{Repeat}
    \item[2.3] \textbf{For all} $t_{j'}, \ t_j +\epsilon_i - 2h_i \leq t_{j'} \leq t_j +\epsilon_i + 2h_i$,
    $$\Omega_{j'} = \Omega_{j'} \cup \{(\mathbf{x_h}+\epsilon_i \mathbf{f} +\alpha_{\epsilon_i,h_i} \mathbf{B}) \cap \mathbf{R}^n_{h_i}, (\mathbf{f},\mathbf{l}) \in \mathrm{FL^+}(\mathbf{x_h})  \}.$$
    \item[2.4] \textbf{Until} all the grid points in $\Omega_{j}$ have been
    visited.
    \end{description}
\item[3] \textbf{Until} $j= J-1.$\\ \\
\end{description}
\textbf{Algorithm 2}\\ \\
\textbf{Initialization}
$\forall t_j, \ T-M \epsilon_i -h_i \leq t_j \leq T + h_i, \ \forall  \mathbf{x_h} \in \Omega_{j}, \ V^{k(\epsilon_i,h_i)}_{\epsilon_i,h_i}(t_j,\mathbf{x_h}) = \{\mathbf{0}\}.$ Otherwise,
$V^{k(\epsilon_i,h_i)}_{\epsilon_i,h_i}(t_j,\mathbf{x_h})  = \emptyset.$\\ \\
Let $j^*$ be the largest index such that $t_{j^*} < T-M \epsilon_i -h_i.$  \\ \\
\textbf{Main loop}\\
\begin{description}
\item[1]\textbf{Set} $j = j^*.$
\item[2]\textbf{Repeat}
    \begin{description}
    \item[2.1] \textbf{Set} $\mathbf{x_h}$ to the first grid point in $\Omega_{j}$ and $A = \emptyset$.
    \item[2.2] \textbf{Repeat}
    \item[2.3] \textbf{For all} $t_{j'}, \ t_j +\epsilon_i - 2h_i \leq t_{j'} \leq t_j +\epsilon_i + 2h_i$,
    $$A = A \cup \{(\epsilon_i\mathbf{l}+ \alpha_{\epsilon_i,h_i} \mathbf{B})\cap \mathbf{R}^p_{h_i}+ V^{k(\epsilon_i,h_i)}_{\epsilon_i,h_i}(t_{j'},(\mathbf{x_h}+\epsilon_i \mathbf{f} +\alpha_{\epsilon_i,h_i} \mathbf{B}) \cap \mathbf{R}^n_{h_i})), (\mathbf{f},\mathbf{l}) \in \mathrm{FL^+}(\mathbf{x_h})\}.$$
    \item[2.4] \textbf{Set}  $V^{k(\epsilon_i,h_i)}_{\epsilon_i,h_i}(t_j,\mathbf{x_h}) = \mathcal{E}(A).$
    \item[2.5] \textbf{Until} all the grid points in $\Omega_{j}$ have been
    visited.
    \end{description}
\item[3] \textbf{Until} $j= 0.$\\
\end{description}

To reduce the size of the set $A$ in \textbf{Algorithm 2}, it is
possible to only keep the Pareto optimal
elements at each  iteration in \textbf{Step 2.3}. This procedure is justified by the two following lemmas.\\

\begin{lemma} \label{lemma:2sets}
Let $S_1$ and $S_2$ be two finite subsets of $\mathbf{R}^{p}$. Then,
$\mathcal{E}(S_1 \cup S_2,P) = \mathcal{E}(S_1 \cup
\mathcal{E}(S_2),P)$.
\end{lemma}
\begin{proof}
Take $\mathbf{z_1} \in \mathcal{E}(S_1 \cup S_2,P)$. Assume for contradiction that $\mathbf{z_1} \notin
\mathcal{E}(S_1 \cup \mathcal{E}(S_2,P),P).$ Then, by external stability,
there exists $\mathbf{z_2} \in \mathcal{E}(S_1 \cup \mathcal{E}(S_2,P),P) \subset S_1 \cup S_2$
such that $\mathbf{z_1} \in \mathbf{z_2} + P \backslash \{\mathbf{0}\}$. But, this contradicts
$\mathbf{z_1} \in \mathcal{E}(S_1 \cup S_2,P)$.\\ \\
Conversely, take $\mathbf{z_1} \in \mathcal{E}(S_1 \cup
\mathcal{E}(S_2),P).$ Assume for contradiction that $\mathbf{z_1}
\notin \mathcal{E}(S_1 \cup S_2,P).$ Then, by external stability,
there exists $\mathbf{z_2} \in \mathcal{E}(S_1 \cup S_2,P) \subset
S_1 \cup S_2$ such that $\mathbf{z_1} \in \mathbf{z_2} + P
\backslash \{\mathbf{0}\}$. Assume that $\mathbf{z_2} \notin S_2$.
Then, necessarily $\mathbf{z_2} \in S_1 \cup \mathcal{E}(S_2,P)$.
But, this contradicts $\mathbf{z_1} \in \mathcal{E}(S_1 \cup
\mathcal{E}(S_2,P),P)$. Assume now that $\mathbf{z_2} \in S_2$.
Then, by external stability, there exists $\mathbf{z_3} \in
\mathcal{E}(S_2,P)$ such that $\mathbf{z_2} \in \mathbf{z_3} + P
\backslash \{\mathbf{0}\}$. Hence, $\mathbf{z_1} \in \mathbf{z_3} +
P \backslash \{\mathbf{0}\}$ with $\mathbf{z_3} \in S_1 \cup
\mathcal{E}(S_2,P).$ But, this again contradicts $\mathbf{z_1} \in
\mathcal{E}(S_1 \cup \mathcal{E}(S_2,P),P)$.
\end{proof}\\

\begin{proposition} \label{lemma:multiplesets}
Let $S_1,\ldots,S_I$ be finite subsets of $\mathbf{R}^{p}$. Then,
$$\mathcal{E}\bigg(\bigcup_{i=1}^I S_i,P\bigg) = E_I,$$ where $E_I$
is recursively defined by $E_1 = \mathcal{E}(S_1,P)$ and the
relation
$$E_{i+1} = \mathcal{E}(S_{i+1}\cup E_i,P).$$
\end{proposition}
\begin{proof}
We proceed by induction. For $I=1$, this is by definition. Assume
now that the relation holds up to $I$. We aim to prove that it
holds for $I+1$, i.e.,
$$\mathcal{E}\bigg(\bigcup_{i=1}^{I+1} S_i\bigg) = E_{I+1}.$$ Apply
Lemma \ref{lemma:2sets} to $\bigcup_{i=1}^{I} S_i$ and $S_{I+1}$.
Then, we get
$$\mathcal{E}\bigg(\bigcup_{i=1}^{I+1} S_i,P\bigg) = \mathcal{E}\bigg(S_{I+1} \cup \mathcal{E}\bigg(\bigcup_{i=1}^{I} S_i,P\bigg),P \bigg).$$
Using the induction assumption, this yields
$$\mathcal{E}\bigg(\bigcup_{i=1}^{I+1} S_i,P\bigg) = \mathcal{E}(S_{I+1} \cup E_I,P) = E_{I+1}.$$
\end{proof}\\

Using Proposition \ref{lemma:multiplesets}, we can change \textbf{Step 2.3} in \textbf{Algorithm 2} to \textbf{Step 2.3'} as follows:\\

\begin{description}
\item[2.1] \textbf{Set} $\mathbf{x_h}$ to the first point in $\Omega_{j}$ and $A = \emptyset$.
\item[2.2]
\item[2.3'] \textbf{For all} $t_{j'}, \ t_j +\epsilon_i - 2h_i \leq t_{j'} \leq t_j +\epsilon_i + 2h_i$,
    $$A = \mathcal{E}(A \cup \{(\epsilon_i\mathbf{l}+ \alpha_{\epsilon_i,h_i} \mathbf{B})\cap \mathbf{R}^p_{h_i}+ V^{k(\epsilon_i,h_i)}_{\epsilon_i,h_i}(t_{j'},(\mathbf{x_h}+\epsilon_i \mathbf{f} +\alpha_{\epsilon_i,h_i} \mathbf{B}) \cap \mathbf{R}^n_{h_i})), (\mathbf{f},\mathbf{l}) \in \mathrm{FL^+}(\mathbf{x_h})\}).$$
     \item[2.4] \textbf{Set}  $V^{k(\epsilon_i,h_i)}_{\epsilon_i,h_i}(t_j,\mathbf{x_h}) = \mathcal{E}(A).$
\item[2.5] \textbf{Until} all the points in $\Omega_{j}$ have been
visited.
\end{description}

\section{Numerical Examples} \label{s:ex}

In this section, the algorithms from \S \ref{s:numalgo} are applied
to a simple class of optimal control problems for which the
set-valued return function $V$ can be obtained analytically. The
convergence of
$V^{k(\epsilon_i,h_i)}_{\epsilon_i,h_i}(-h_i,\mathbf{x_0})$ towards
$V(0,\mathbf{x_0})$ is  investigated. Recall that incrementing $i$
by 1 means dividing the time step discretization by 2 and
the state step discretization by 4.

\subsection{Description} \label{subsec:description}

Consider the following simple autonomous biobjective ($p=2$) optimal control problem. The one-dimensional ($n=1$) dynamics
is simply
\begin{displaymath}
 \dot{x}(s) = u(s), \ s \in [0,T],
\end{displaymath}
with $U = \{-1,1\}$ and initial condition $x_0$. The cost of a
trajectory $x(\cdot)$ over $I$ is given by
\begin{displaymath}
\mathbf{J}_1(0,x_0,u(\cdot))  =  \displaystyle \int_{0}^{T} P(x(s)) u(s) \ \mathrm{d}s \ \mathrm{and} \ \mathbf{J}_2(0,x_0,u(\cdot))  = \displaystyle \int_{0}^{T} u(s) \ \mathrm{d}s, \\
\end{displaymath}
where $P(\cdot)$ is a given polynomial.\\

The objective space $Y(0,x_0)$ for the problem above can be easily
determined. Let $\alpha = \mu(\{t \in I, \ u(t) = 1 \})/T$ where
$\mu(\cdot)$ denotes the Lebesgue measure, then
$$ \mathbf{J}_2(0,x_0,u(\cdot)) = \int_{0}^{T} u(s) \ \mathrm{d}s = \alpha T  -(1-\alpha)T = (2\alpha-1)T, $$
and $$\mathbf{J}_1(0,x_0,u(\cdot)) = \int_{0}^{T} P(x(s)) u(s) \
\mathrm{d}s = \int_{0}^{T} P(x(s)) \dot{x}(s) \ \mathrm{d}s  =
[Q(x(t))]_0^T,$$ where $Q(\cdot)$ is an antiderivative of
$P(\cdot)$. With $$x(T) = x(0) + \int_{0}^{T} u(s) \ \mathrm{d}s =
x_0 + (2\alpha-1)T,$$ and defining $\delta = (2\alpha-1)T$, we get
\begin{displaymath}
\mathbf{J}_1(0,x_0,u(\cdot))  =  Q(x_0+\delta) - Q(x_0) \
\mathrm{and} \ \mathbf{J}_2(0,x_0,u(\cdot))  =  \delta.
\end{displaymath}
The set $Y(0,x_0)$ can finally be obtained by varying $\delta$ between $-T$ and $T$. \\

It is possible to derive a systematic procedure that gives an
interval ($\subset [-T,T]$) such that for all $\delta$ in this
interval, the corresponding element $(\mathbf{J}_1(0,x_0,u(\cdot)),
\mathbf{J}_2(0,x_0,u(\cdot))) = (Q(x_0+\delta) - Q(x_0),\delta)$ in
the objective space  is a Pareto optimal element. We will not
present this procedure here, and therefore assume  that the Pareto
optimal set $V(0,x_0)$ is known.

\subsection{Results}

We consider four different polynomials $P(\cdot)$ and initial
conditions $x_0$ with $T=0.5$, which yield four problems (MOC1),
(MOC2), (MOC3), and (MOC4). We have chosen these polynomials such that
the Pareto optimal sets $V(0,x_0)$ present different characteristics.
For each problem and for  $i=3$, $i=4$, and $i=5$, we compute
$V^{k(\epsilon_i,h_i)}_{\epsilon_i,h_i}(-h_i,x_0)$ using \textbf{Algorithm 1} and \textbf{Algorithm 2}
from  \S \ref{s:numalgo}, provide the cardinality of $V^{k(\epsilon_i,h_i)}_{\epsilon_i,h_i}(-h_i,x_0)$, i.e.,
$|V^{k(\epsilon_i,h_i)}_{\epsilon_i,h_i}(-h_i,x_0)|$, calculate the
Hausdorff distance \cite[p.~365]{SVA_book}
$\mathcal{H}(V^{k(\epsilon_i,h_i)}_{\epsilon_i,h_i}(-h_i,x_0),V(0,x_0))$
between $V^{k(\epsilon_i,h_i)}_{\epsilon_i,h_i}(-h_i,x_0)$
and $V(0,x_0)$, and finally generate a ''normalized'' Hausdorff distance
$$ \overline{\mathcal{H}}(V^{k(\epsilon_i,h_i)}_{\epsilon_i,h_i}(-h_i,x_0),V(0,x_0)) = \frac{\mathcal{H}(V^{k(\epsilon_i,h_i)}_{\epsilon_i,h_i}(-h_i,x_0),V(0,x_0))}{\mathcal{H}(V^{k(\epsilon_3,h_3)}_{\epsilon_3,h_3}(-h_3,x_0),V(0,x_0))}.$$
Our results are summarized in Tables \ref{tab-firstP}, \ref{tab-secondP},  \ref{tab-thirdP}, and \ref{tab-fourthP} and also
in Figures \ref{fig-firstP}, \ref{fig-secondP},  \ref{fig-thirdP}, and \ref{fig-fourthP}. Tables \ref{tab-firstP}, \ref{tab-secondP},  \ref{tab-thirdP}, and \ref{tab-fourthP} also contain the size of the corresponding problem. The first number corresponds to the total number of grid points, i.e., the cardinality of all the sets  $\Omega_j$. The second number corresponds to the total number of successors for all
the grid points, where a successor to a grid point $\mathbf{x_h}$ is defined as a grid point that can be reached
from $\mathbf{x_h}$. Using Graph Theory terminology, the size of a problem would correspond to the number of nodes and
 vertices respectively. Figures \ref{fig-firstP}, \ref{fig-secondP},  \ref{fig-thirdP}, and \ref{fig-fourthP} display the objective space,
the Pareto optimal set, and the approximate Pareto optimal sets for $i=3$, $i=4$, and $i=5$ or each problem.\\

\begin{description}
  \item[(MOC1)] $P(x) = x-1, \ x_0 = 1$. The set $Y(0,x_0)+\mathbf{R}_+^2$ is convex.
  Hence, it is possible to obtain every element of $V(0,x_0)$
  using the weighting method. 
  \item[(MOC2)] $P(x) = -x+1, \ x_0 = 1.5$. The set $Y(0,x_0)-\mathbf{R}_+^2$ is convex. Only the two Pareto optimal elements for $\delta = -T$
  and $\delta = T$ can be obtained using the weighting method. 
  \item[(MOC3)] $P(x) = -2x^3 -15/4 x^2 + 2/75x+ 1/5, \ x_0 = 0$. The set $Y(0,x_0) + \mathbf{R}_+^2$ is nonconvex and the set $V(0,x_0)$ is
  nonconnected. More precisely, $V(0,x_0)$ is the union of two sets. 
  \item[(MOC4)] $P(x) = -3/2x -1/8, \ x_0 = 0$. This problem is similar to (MOC3).\\ 
\end{description}

\begin{remark} To reduce the size of the problems,
we have proceeded to some simplifications in \emph{\textbf{Algorithm 1}} and \emph{\textbf{Algorithm 2}}. First, we have individually computed $\alpha_{\epsilon_i,h_i}$ for the dynamics $\mathbf{f}(\cdot,\cdot)$ and each component of the  running cost $\mathbf{L}(\cdot,\cdot)$. Second,  we have reduced  the interval $[t_j +\epsilon_i - 2h_i,t_j +\epsilon_i + 2h_i]$ to the single time $t_j +\epsilon_i - 2h_i$. Finally, for any given $\mathbf{l}$, we have reduced the set $(\epsilon_i\mathbf{l}+ \alpha_{\epsilon_i,h_i} \mathbf{B})\cap \mathbf{R}^p_{h_i}$ to a single element, i.e.,
the closest lattice element to $\epsilon_i\mathbf{l}$, which somehow corresponds to setting $\alpha_{\epsilon_i,h_i} =0$ for the running cost. Hence, we have set $M$ to $\max \{1,M_{\mathbf{f}}\} = 1.$
\end{remark}

\begin{table}[h!]
\renewcommand\arraystretch{1.2}
\begin{center}
 \caption{Results for \emph{(MOC1)}.}
\label{tab-firstP}
\begin{tabular}{|c|c|c|c|}
\hline
& i=3 &  i=4 & i=5 \\
\hline Size & (306,5897) & (1630,65093) & (10422,856445)\\
\hline $|V^{k(\epsilon_i,h_i)}_{\epsilon_i,h_i}(-h_i,x_0)|$
 & 10 & 33 & 130\\
\hline
$\mathcal{H}(V^{k(\epsilon_i,h_i)}_{\epsilon_i,h_i}(-h_i,x_0),V(0,x_0))$
&  0.091227 & 0.046550 & 0.022605\\
\hline
$\overline{\mathcal{H}}(V^{k(\epsilon_i,h_i)}_{\epsilon_i,h_i}(-h_i,x_0),V(0,x_0))$
&  1 & 0.5103 &  0.2478\\
\hline
 \end{tabular}
\end{center}
\end{table}

\begin{table}[h!]
\renewcommand\arraystretch{1.2}
\begin{center}
 \caption{Results for \emph{(MOC2)}.}
\label{tab-secondP}
\begin{tabular}{|c|c|c|c|}
\hline
& i=3 &  i=4 & i=5 \\
\hline Size & (306,10961)  & (1630,132125) & (10422,1826357)\\
\hline $|V^{k(\epsilon_i,h_i)}_{\epsilon_i,h_i}(-h_i,x_0)|$
& 34 & 130 & 514\\
\hline
$\mathcal{H}(V^{k(\epsilon_i,h_i)}_{\epsilon_i,h_i}(-h_i,x_0),V(0,x_0))$
& 0.051067 & 0.033192  &  0.016627 \\
\hline
$\overline{\mathcal{H}}(V^{k(\epsilon_i,h_i)}_{\epsilon_i,h_i}(-h_i,x_0),V(0,x_0))$
& 1 & 0.65  &  0.3256 \\
\hline
 \end{tabular}
\end{center}
\end{table}

\begin{table}[h!]
\renewcommand\arraystretch{1.2}
\begin{center}
 \caption{Results for \emph{(MOC3)}.}
\label{tab-thirdP}
\begin{tabular}{|c|c|c|c|}
\hline
& i=3 &  i=4 & i=5 \\
\hline Size & (306,6529) & (1630,66613) & (10422,834285)\\
\hline $|V^{k(\epsilon_i,h_i)}_{\epsilon_i,h_i}(-h_i,x_0)|$
& 3 & 21 & 99\\
\hline
$\mathcal{H}(V^{k(\epsilon_i,h_i)}_{\epsilon_i,h_i}(-h_i,x_0),V(0,x_0))$
& 0.765685 &  0.054420 & 0.035360 \\
\hline
$\overline{\mathcal{H}}(V^{k(\epsilon_i,h_i)}_{\epsilon_i,h_i}(-h_i,x_0),V(0,x_0))$
& 1 &  0.0711 & 0.0462 \\
\hline
 \end{tabular}
\end{center}
\end{table}

\begin{table}[h!]
\renewcommand\arraystretch{1.2}
\begin{center}
 \caption{Results for \emph{(MOC4)}.}
\label{tab-fourthP}
\begin{tabular}{|c|c|c|c|}
\hline
& i=3 &  i=4 & i=5 \\
\hline Size & (306,7553) & (1630,85213) & (10422,1134221)\\
\hline $|V^{k(\epsilon_i,h_i)}_{\epsilon_i,h_i}(-h_i,x_0)|$
& 9 & 33  & 129\\
\hline
$\mathcal{H}(V^{k(\epsilon_i,h_i)}_{\epsilon_i,h_i}(-h_i,x_0),V(0,x_0))$
& 0.033857 & 0.028646 & 0.014031\\
\hline
$\overline{\mathcal{H}}(V^{k(\epsilon_i,h_i)}_{\epsilon_i,h_i}(-h_i,x_0),V(0,x_0))$
& 1 & 0.8461 & 0.4144\\
\hline
 \end{tabular}
\end{center}
\end{table}

\begin{figure}[!h]
\centering
\includegraphics[height=9cm]{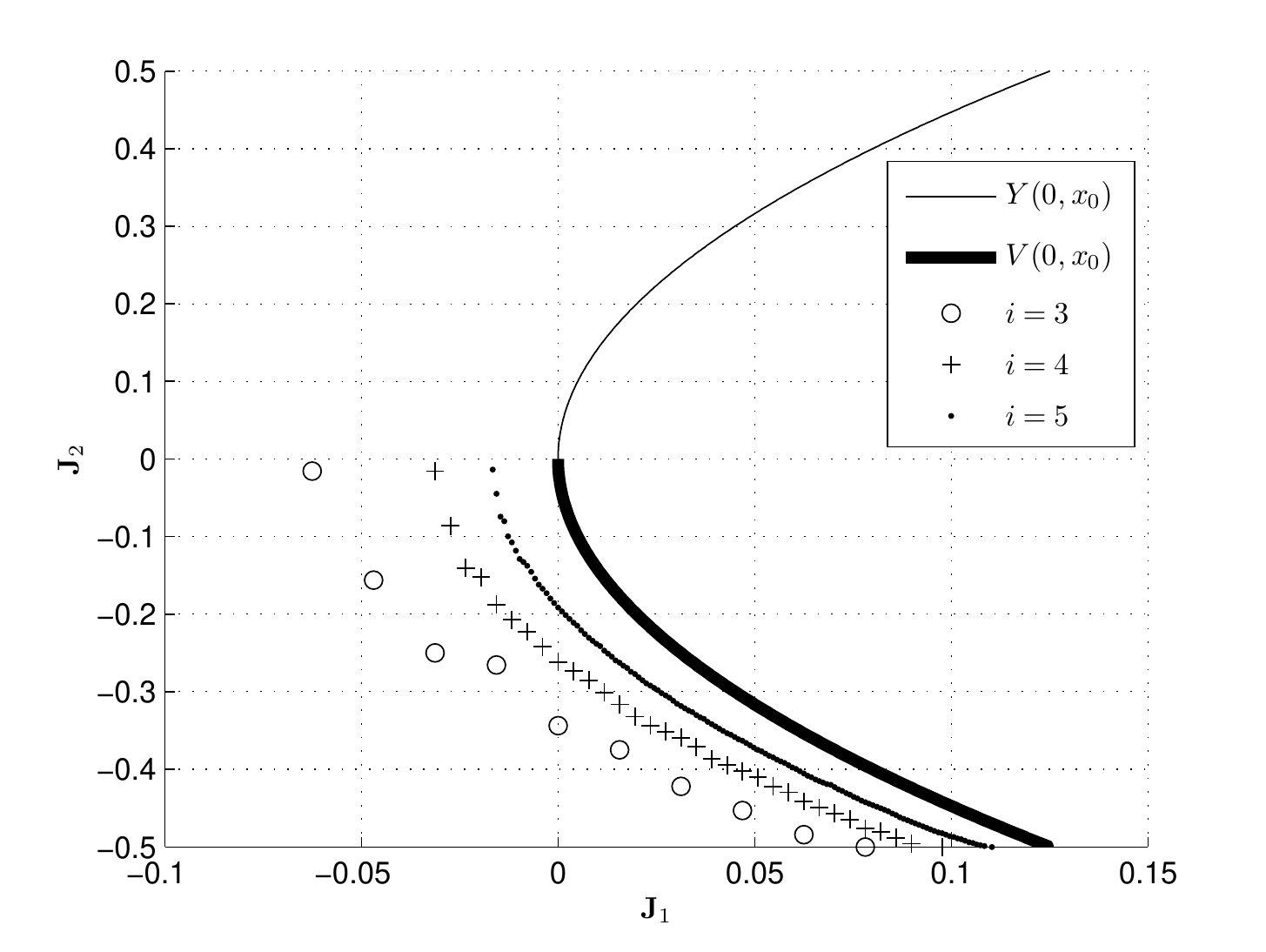}
   \caption{\emph{(MOC1)}: Objective space (plain line), Pareto optimal set $V(0,x_0)$ (bold line) and approximate Pareto optimal set $V^{k(\epsilon_i,h_i)}_{\epsilon_i,h_i}(-h_i,x_0)$ for  $i=3$ (o), $i=4$ (+), and $i=5$ ($\cdot$).}
   \label{fig-firstP}
\end{figure}

\begin{figure}[!h]
\centering
\includegraphics[height=9cm]{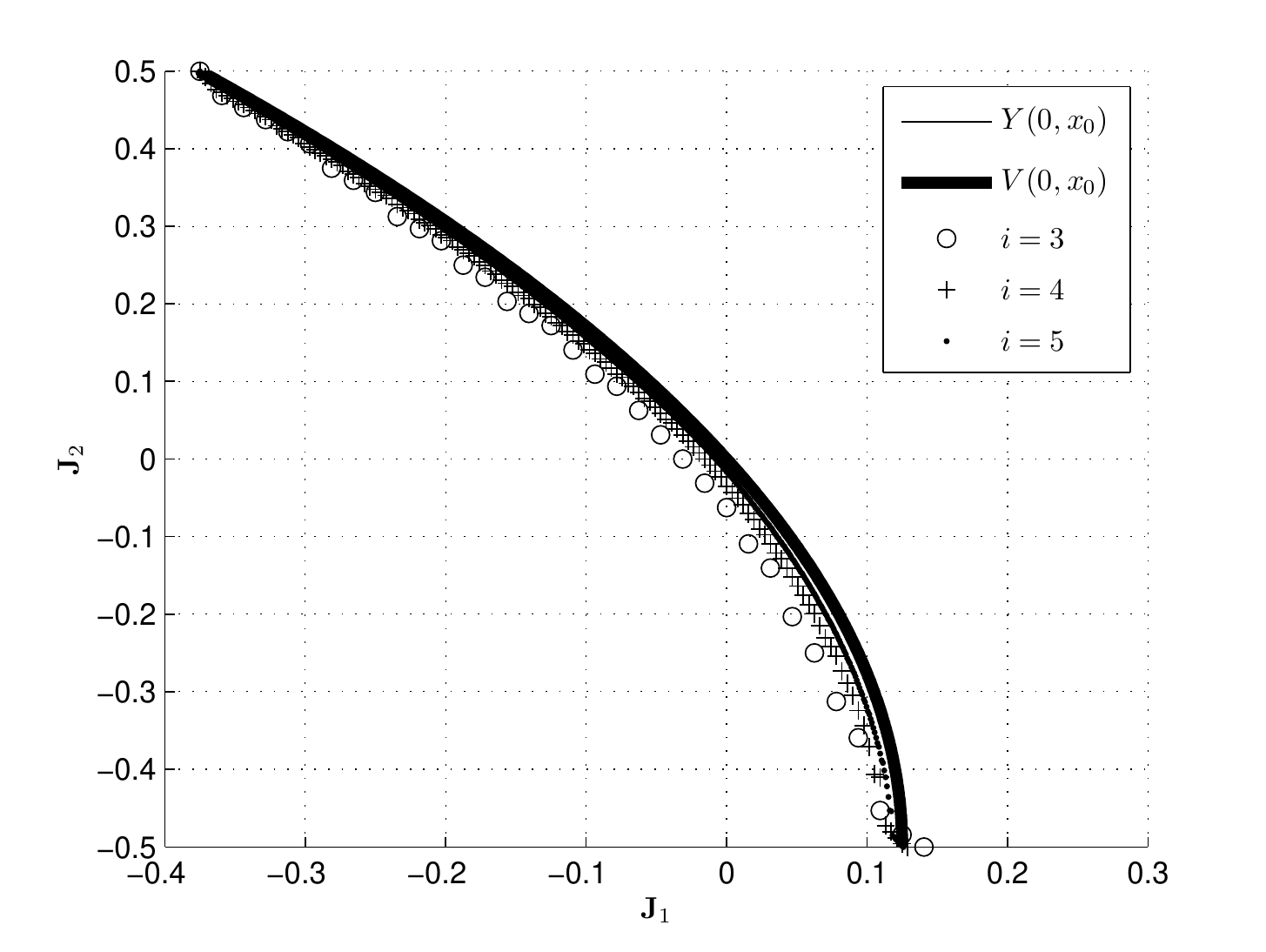}
   \caption{\emph{(MOC2)}: Objective space (plain line), Pareto optimal set $V(0,x_0)$ (bold line) and approximate Pareto optimal set $V^{k(\epsilon_i,h_i)}_{\epsilon_i,h_i}(-h_i,x_0)$ for  $i=3$ (o), $i=4$ (+), and $i=5$ ($\cdot$).}
   \label{fig-secondP}
\end{figure}

\begin{figure}[!h]
\centering
\includegraphics[height=9cm]{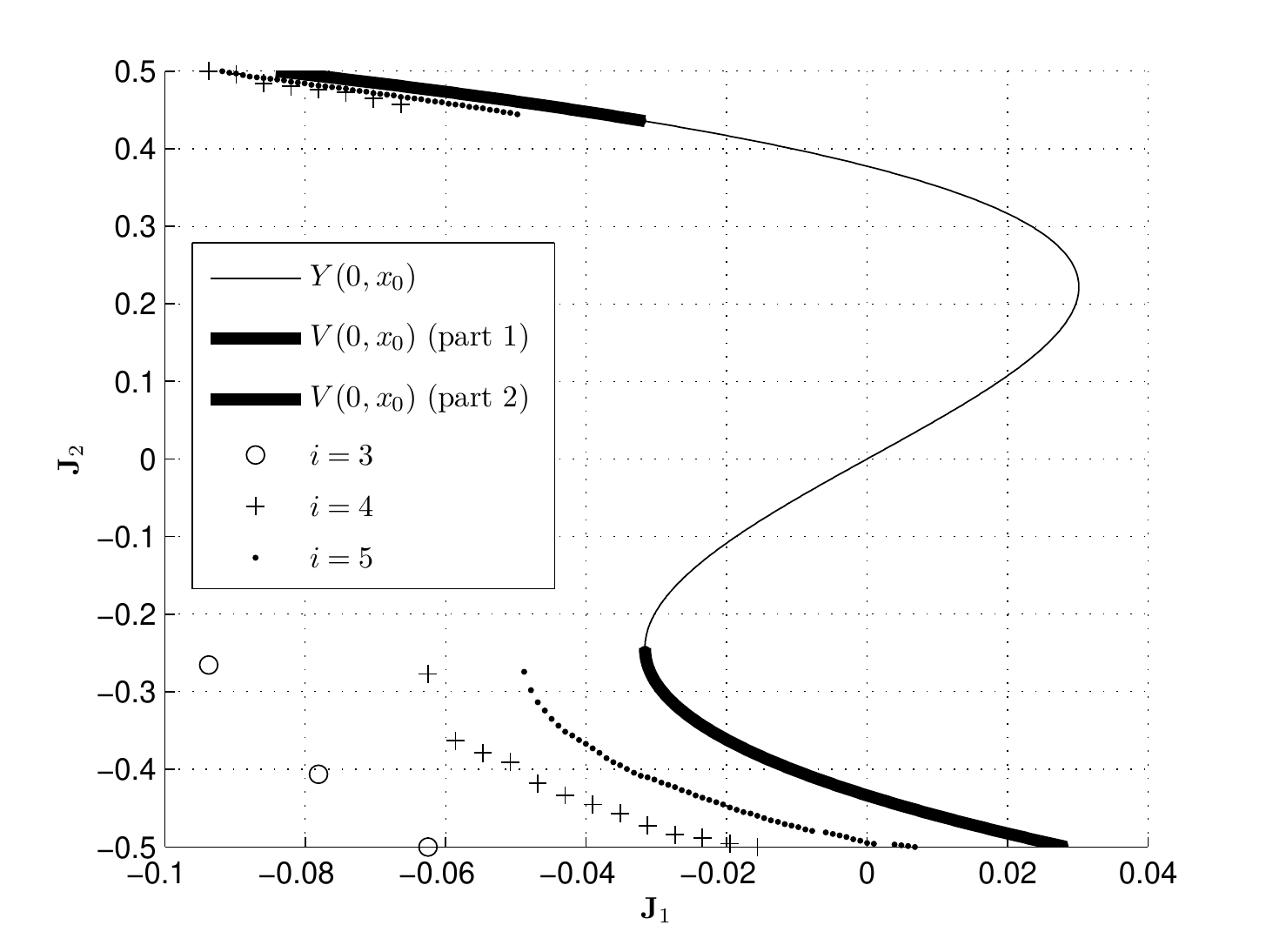}
   \caption{\emph{(MOC3)}: Objective space (plain line), Pareto optimal set $V(0,x_0)$ (bold lines) and approximate Pareto optimal set $V^{k(\epsilon_i,h_i)}_{\epsilon_i,h_i}(-h_i,x_0)$ for  $i=3$ (o), $i=4$ (+), and $i=5$ ($\cdot$).}
   \label{fig-thirdP}
\end{figure}

\begin{figure}[!h]
\centering
\includegraphics[height=9cm]{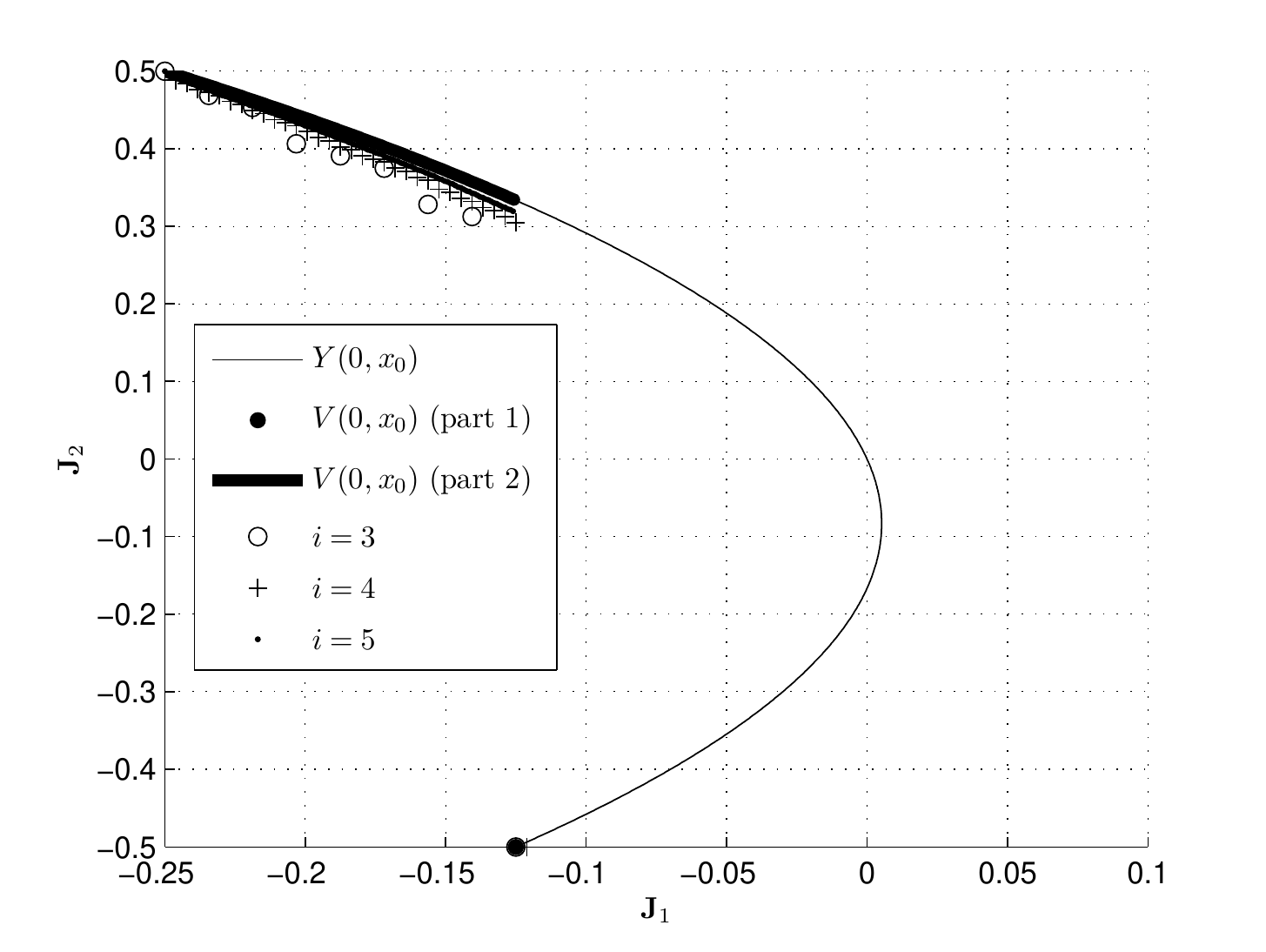}
   \caption{\emph{(MOC4)}: Objective space (plain line), Pareto optimal set $V(0,x_0)$ (bold lines) and approximate Pareto optimal set $V^{k(\epsilon_i,h_i)}_{\epsilon_i,h_i}(-h_i,x_0)$ for  $i=3$ (o), $i=4$ (+), and $i=5$ ($\cdot$).}
   \label{fig-fourthP}
\end{figure}

\clearpage

\subsection{Discussion}

Because the dynamics and the final time $T$ are the same for the four problems, it is normal that
the total number of grid points remains the same. On the other hand, the running cost is different
for the four problems, hence for a given grid point $\mathbf{x_h}$, the sets $\mathrm{FL^+}(\mathbf{x_h})$
 differ, which explains why the  number of successors varies between the four problems.\\

The large value, i.e., 0.765685,  in Table \ref{tab-thirdP} for (MOC3) comes from the fact that the approximate
Pareto optimal set is not able for $i=3$ to capture the upper part of the Pareto optimal set. However, as $i$ increases,
the approximate Pareto optimal set now captures the upper part of the Pareto optimal set, and as result, the
Hausdorff distance $\mathcal{H}(V^{k(\epsilon_i,h_i)}_{\epsilon_i,h_i}(-h_i,x_0),V(0,x_0))$ decreases considerably.\\

As $i$ increases, as expected, for the four problems, a better approximation of the Pareto optimal set is obtained.
The proposed approach also works well regardless whether the set $Y(0,x_0)+\mathbf{R}_+^p$ is convex
or not.

\section{Conclusion} \label{s:con}

In this paper, we have  derived a convergent approximation of the Pareto optimal
set for finite-horizon multiobjective optimal control problems. Several techniques
such as domain decomposition \cite{OptConTh8_book} or dynamic grid refinement \cite{Cardaliaguet99_Via}
could be considered to improve the computational complexity of the proposed approach.  Let us also mention  the idea of clustering recently proposed in
\cite{Guigue10_1}, which consists of keeping only a subset (with fixed cardinality) of the approximate
Pareto optimal set at each grid point during the resolution of the multiobjective dynamic programming equation (\ref{eq:recursive_1})-(\ref{eq:recursive_2}).\\

A direct extension to this work would be to consider the general
case of a pointed closed convex cone $P$ instead of the nonnegative
orthant $\mathbf{R}_+^p$ and constraints on the state. Also, the
proposed approach could be considered for other classes of optimal
control problems, such as multiobjective exit-time optimal control
problems \cite{MOC_num_2}.

\bibliography{biblio}

\begin{thebibliography}{10}

\bibitem{Book_Via}
{\sc J.-P. Aubin}, {\em Viability theory}, Birkhauser, Boston, 1991.

\bibitem{SVA_book}
{\sc J.-P. Aubin and H.~Frankowska}, {\em Set-Valued Analysis}, Birkhauser,
  Boston, 1990.

\bibitem{OptConTh8_book}
{\sc M.~Bardi and I.~Capuzzo-Dolcetta}, {\em Optimal Control and Viscosity
  Solutions of Hamilton-Jacobi-Bellman Equations}, Birkhauser, Boston, 1997.

\bibitem{Cardaliaguet99_Via}
{\sc P.~Cardaliaguet, M.~Quincampoix, and P.~Saint-Pierre}, {\em Set-valued
  numerical analysis for optimal control and differential games}, Stochastic
  and differential games : Theory and numerical methods. Annals of the
  international Society of Dynamic Games, M. Bardi, T.E.S. Raghavan, T.
  Parthasarathy Eds., Birkhauser, Boston,  (1999), pp.~177--247.

\bibitem{Cardaliaguet00_Via}
\leavevmode\vrule height 2pt depth -1.6pt width 23pt, {\em Numerical schemes
  for discontinuous value functions of optimal control}, Set-Valued Analysis, 8
  (2000), pp.~111--126.

\bibitem{Cardaliaguet07_Via}
\leavevmode\vrule height 2pt depth -1.6pt width 23pt, {\em Differential games
  through viability theory: Old and recent results}, Advances in Dynamic Game
  Theory. Annals of the international Society of Dynamic Games, S. Jorgensen,
  M. Quincampoix, T.~L. Vincent, T. Basar Eds., Birkhauser, Boston,  (2007),
  pp.~3--35.

\bibitem{MOC_appl_6}
{\sc V.~Coverstone-Carroll, J.~W. Hartmann, and W.~J. Mason}, {\em Optimal
  multi-objective low-thrust spacecraft trajectories}, Comput. Methods Appl.
  Mech. Engrg., 186 (2000), pp.~387--402.

\bibitem{Book_EA}
{\sc K.~Deb}, {\em Multi-objective optimization using evolutionary algorithms},
  John Wiley \& Sons, Chichister, 2001.

\bibitem{MOC_appl_7}
{\sc P.~J. Fleming and R.~C. Purshouse}, {\em Evolutionary algorithms in
  control systems engineering: a survey}, Control Engineering Practice, 10
  (2002), pp.~1223–--1241.

\bibitem{Guigue10_1}
{\sc A.~Guigue}, {\em An approximation method for multiobjective optimal
  control problems application to a robotic trajectory planning problem.},
  Submitted to Optim. Eng.,  (2010).

\bibitem{Guigue11}
\leavevmode\vrule height 2pt depth -1.6pt width 23pt, {\em Set-valued return
  function and generalized solutions for multiobjective optimal control
  problems (moc)}, Submitted to SIAM J. Control Optim.,  (2011).

\bibitem{Guigue09}
{\sc A.~Guigue, M.~Ahmadi, M.~J.~D. Hayes, and R.~G. Langlois}, {\em A discrete
  dynamic programming approximation to the multiobjective deterministic finite
  horizon optimal control problem}, SIAM J. Control Optim., 48 (2009),
  pp.~2581--2599.

\bibitem{Guigue10}
{\sc A.~Guigue, M.~Ahmadi, R.~G. Langlois, and M.~J.~D. Hayes}, {\em Pareto
  optimality and multiobjective trajectory planning for a 7-dof redundant
  manipulator}, IEEE Transactions on Robotics, 26 (2010), pp.~1094--1099.

\bibitem{MOC_appl_5}
{\sc B.-Z. Guo and B.~Sun}, {\em Numerical solution to the optimal feedback
  control of continuous casting process}, J. Glob. Optim., 39 (1998),
  pp.~171–--195.

\bibitem{MOC_num_2}
{\sc A.~Kumar and A.~Vladimirsky}, {\em An efficient method for multiobjective
  optimal control and optimal control subject to integral constraints}, J.
  Comp. Math., 28 (2010), pp.~517--551.

\bibitem{NonlinearMOO2_book}
{\sc K.~M. Miettinen}, {\em Nonlinear Multiobjective Optimization}, Kluwer
  Academic Publishers, Boston, 1999.

\bibitem{NonlinearMOO1_book}
{\sc Y.~Sawaragi, H.~Nakayama, and T.~Tanino}, {\em Theory of Multiobjective
  Optimization}, Academic Press, Inc., Orlando, 1985.

\bibitem{Tanino88_moo}
{\sc T.~Tanino}, {\em Sensitivity analysis in multiobjective optimization}, J.
  Optim. Theory Appl., 56 (1988), pp.~479--499.

\bibitem{OptConTh5_book}
{\sc R.~Vinter}, {\em Optimal Control}, Birkauser, Boston, 2000.

\bibitem{Yu74_moo}
{\sc P.~L. Yu}, {\em Cone convexity, cone extreme points, and nondominated
  solutions in decision problems with multiobjectives}, J. Optim. Theory Appl.,
  14 (1974), pp.~319--377.

\end{thebibliography}
\bibliographystyle{siam}

\end{document}